\documentclass{daj}
\usepackage{hyperref}
\usepackage{enumerate}
\usepackage{comment}

\makeatletter

\@namedef{subjclassname@2010}{

  \textup{2010} Mathematics Subject Classification}

\usepackage{amssymb, amsmath,amsthm}
\usepackage{mathrsfs}
\newtheorem{thm}{Theorem}[section]

\newtheorem{prop}[thm]{Proposition}

\newtheorem{cor}[thm]{Corollary}
\newtheorem{lem}[thm]{Lemma}
\theoremstyle{definition}
\newtheorem{rem}[thm]{Remark}

\newcommand{\ra}{\rightarrow}
\newcommand{\bk}{\backslash}
\newcommand{\mc}{\mathcal}

\newcommand{\mb}{\mathbb}
\newcommand{\sg}{\sigma}

\renewcommand{\ss}{\substack}
\newcommand{\llf}{\left\lfloor}

\newcommand{\e}{\varepsilon}
\newcommand{\rrf}{\right\rfloor}

\renewcommand{\bar}{\overline}

\dajAUTHORdetails{%
  title = {On Equal Consecutive Values of Multiplicative Functions}, %% please capitalize all significant words
  author = {Alexander P. Mangerel},
  plaintextauthor = {Alexander P. Mangerel},
  plaintexttitle = {On Equal Consecutive Values of Multiplicative Functions}, %%  title without math or LaTeX
  runningtitle = {Equal consecutive values of multiplicative functions}, 
  runningauthor = {Alexander P. Mangerel},
  copyrightauthor = {P. Erd\H{o}s, J. H{\aa}stad, L. Lov\'asz, and A. C-C. Yao},
  keywords = {keyword, keyword, etc.},
}   %%% END \dajAUTHORdetails

\dajEDITORdetails{%
   year={2024},
   number={12},
   received={19 June 2023},   % received date: example: 7 January 2017
   revised={4 January 2024},    % Optional revised date (you may comment it out)
   published={11 November 2024},  % published date
   doi={10.19086/da.125450},       % XXX = number of paper, e.g. da006 for paper#6
}   %%% END \dajEDITORdetails

\begin{document}

\begin{frontmatter}[classification=text]
%% EDITOR: this will force the keywords to appear right after the Abstract.
%%   If the abstract is too long and would force the keywords off the
%%   front page, please comment out % [classification=text] above
%%   This way the keywords will be floated on the bottom of the first page
%%   even though the Abstract spills over to the next page.

%%% AUTHOR: Title goes here.  This line is optional.  You must use it
%%   if title has footnote attached or requires nontrivial typesetting,
%%   e.g., inclusion of linebreaks to force nice layout.
\title{On Equal Consecutive Values of Multiplicative Functions} %% please capitalize all significant words

%%% AUTHOR:
%%% List all authors. If you wish, place grant acknowledgements in \thanks.
%%% In brackets include a short tag for each author.
\author[APM]{Alexander P. Mangerel}
%%% AUTHOR: Abstract goes here
\begin{abstract}
Let $f: \mb{N} \ra \mb{C}$ be a multiplicative function for which
$$
\sum_{p : \, |f(p)| \neq 1} \frac{1}{p} = \infty.
% \quad \sum_{p : \, f(p) = 0} \frac{1}{p} < \infty.
$$
We show under this condition alone
%(and no further assumptions on the growth of $f$) 
that for any integer $h \neq 0$ the set 
$$
\{n \in \mb{N} : f(n) = f(n+h) \neq 0\}
$$
has logarithmic density 0. We also prove a converse result, along with an application
%. provide generalisations of this result and a converse result, along with an application 
to the Fourier coefficients of holomorphic cusp forms. 
%In particular, we make no assumptions on the growth of $f$. 

The proof involves analysing the value distribution of $f$ using the compositions $|f|^{it}$,
% and $\chi(f)$ for varying, suitably chosen $t \in \mb{R}$ and Dirichlet characters $\chi$, both of which are multiplicative functions taking values in the closed unit disc. 
relying crucially on various applications of Tao's theorem on logarithmically-averaged correlations of non-pretentious multiplicative functions. Further key inputs arise from the inverse theory of sumsets in continuous additive combinatorics.
\end{abstract}
\end{frontmatter}

%%% AUTHOR: body of paper starts here

\section{Introduction}
There is an expected general phenomenon in analytic number theory that the factorisations of additively coupled positive integers $n$ and $n+h$, $h \geq 1$, are \emph{statistically independent}, in the sense that for \emph{very few} $n \in \mb{N}$ do we expect the features (number of primes, size of prime factors etc.) of the prime factorisation of $n$ to influence the corresponding features of the factorisation of $n+h$. To investigate this phenomenon we are motivated to study the joint value distribution, as $n$ varies, of pairs $(f(n),f(n+h))$, where $f: \mb{N} \ra \mb{C}$ is a multiplicative function.
%, thus sensitive to the factorisations of integers. 

Bounding the size of the set $\{n \leq x : f(n) = f(n+h)\}$ as $x \ra \infty$, for $h \in \mb{N}$ and $f$ a multiplicative function, is a natural problem, for which some precedent in the literature already exists. For example, when $f = d$, the divisor function,
%$\omega(n) := |\{p \text{ prime} : p |n\}$, and $z \in \mb{C}$. When $z > 0$ this is equivalent to a problem about coincidences $\omega(n) = \omega(n+1)$. 
Erd\H{o}s, Pomerance and Sark\H{o}zy \cite{EPS} proved using sieve theoretic arguments that
$$
|\{n \leq x: d(n) = d(n+1)\}| \ll \frac{x}{\sqrt{\log\log x}},
$$
and thus $d(n) = d(n+1)$ occurs on a set of natural\footnote{The \emph{density} of a set $A \subseteq \mb{N}$ is defined by $\lim_{x \ra \infty} |A \cap [1,x]|/x$, provided the limit exists. The \emph{upper density} of $f$ is given by the same definition with $\limsup$ in place of $\lim$.} density $0$. In a separate paper \cite{EPS2}, the same authors obtain stronger upper bounds for the number of solutions $n \leq x$ to $\sg(n) = \sg(n+1)$ and to $\phi(n) = \phi(n+1)$, where $\sg$ and $\phi$ are, respectively, the sum-of-divisors function and Euler's phi function. The explicit nature of these three examples is significant in the arguments of \cite{EPS} and \cite{EPS2}.

In a recent paper \cite{heckeKM}, Klurman and the author considered a different example of interest, when $(f(n))_n$ corresponds to the sequence of Fourier coefficients of a holomorphic cusp form. More precisely, for $k \geq 2$ let $\phi$ be a weight $k$, arithmetically normalised, holomorphic cuspidal eigenform without complex multiplication for $\text{SL}_2(\mb{Z})$.
% of some weight $k \geq 2$. 
%particular instance of this problem in which $(f(n))_n$ corresponds to the sequence of Fourier coefficients of such a (correctly normalised) form $\phi$. 
Writing $\phi$ in its Fourier expansion
$$
\phi(z) = \sum_{n \geq 1} a_{\phi}(n)e(nz), \quad \text{Im}(z) > 0,
$$
normalised so that $a_{\phi}(1) = 1$, it was shown among other things that each of the sets
% of solutions $n,m \in \mb{N}$ to the inequalities
$$
\{n \in \mb{N} : |a_{\phi}(n)| < |a_{\phi}(n+h)|\}, \quad \{m \in \mb{N} : |a_{\phi}(m+h)| < |a_{\phi}(m)|\},
$$
%\footnote{We say that $A$ has \emph{logarithmic density 0} if $\lim_{x \ra \infty} \tfrac{1}{\log x} \sum_{n \in A \cap [1,x]} \tfrac{1}{n} = 0$.}
has positive density, and in fact it follows from \cite[Thm. 1.7]{heckeKM}
%\footnote{The result given in \cite{heckeKM} shows that the upper density of the set of $n$ satisfying each of the above inequalities is $\geq 1/2$. Since these upper densities may be witnessed on subsequences that do not coincide, however, we cannot conclude unconditionally that each of these occurs on a density $1/2$ set. The corresponding statement for the \emph{natural} density of solutions with $a_{\phi}$ replaced by $|a_{\phi}|$ does hold, though.} 
that for every $h \geq 1$ the set of \emph{non-vanishing} coincidences\footnote{The non-vanishing assumption here is necessary given our current knowledge about the vanishing of Fourier coefficients of cusp forms. In particular, if there exists even a single \emph{odd} integer $d_0$ for which $a_{\phi}(d_0) = 0$ then a positive density sequence of coincidences $a_{\phi}(n+d_0) = a_{\phi}(n) = 0$ can be found. 
%the stated result is false if we do not exclude the set of vanishing coefficients here. 
Indeed, it suffices to choose the set $\{md_0 : m \in \mb{N}, \ (m(m+1),d_0) = 1\}$. Ruling out the existence of such $d_0$ for the Ramanujan $\tau$-function, for example, would require proving a well-known conjecture due to Lehmer.}
$$
\{n \in \mb{N}: \, a_{\phi}(n) = a_{\phi}(n+h) \neq 0\}
$$
%that is, excluding possible vanishing Fourier coefficients,
has natural density $0$.
%\emph{natural density $0$}, i.e., that
%\begin{equation}\label{eq:logDensaPhi}
%\lim_{x \ra \infty} \frac{1}{\log x} \sum_{\ss{n \leq x \\ a_{\phi}(n)a_{\phi}(n+h) \neq 0}} \frac{1_{a_{\phi}(n) = a_{\phi}(n+h)}}{n} = 0.
%\end{equation}
This can be understood heuristically from the fact that if $\phi$ has weight $k \geq 2$ then $a_{\phi}(n)$ generically grows like $n^{(k-1)/2 + o(1)}$ (due to work of Deligne), and thus the likelihood of coincidences among non-zero values in such a large set is rather small. 

More generally, we might speculate (as is done in some form on \cite[p. 2]{heckeKM}) that the same conclusion about the paucity of coincidences $f(n) = f(n+h)$ ought to be true for \emph{any} unbounded multiplicative function $f: \mb{N} \ra \mb{C}$ such that $|f(p)| \neq 1$ for most primes $p$ (if, conversely, $|f(p)| = 1$ for many $p$ this would account for numerous coincidences among non-zero values in $(f(m))_m$, as indicated by Proposition \ref{prop:partConv} below).

%A priori, testing this speculation is not so straightforward, given that
%there are few general techniques in the theory of multiplicative functions that are applicable to unbounded functions growing faster than a power of the divisor function $d(n)$, say. This makes analysing the behaviour of functions with polynomial growth, for instance, not immediately accessible by
%%\footnote{The aforementioned work of Deligne does allow us to say that the normalised Fourier coefficients $\lambda_{\phi}(n) := a_{\phi}(n)n^{-(k-1)/2}$ gives a multiplicative function with $|\lambda_{\phi}(n)| \leq d(n)$, so existing techniques in the literature might apply more readily to this collection.}
%existing means.
In this paper, we consider the general problem of bounding the number of solutions in positive integers $n$ to the equation
$$
f(n) = f(n+h) \neq 0,
$$
where $f: \mb{N} \ra \mb{C}$ is an \emph{arbitary}
%n integer-valued 
multiplicative function, provided only that $|f(p)| \neq 1$ sufficiently often (in a precise sense). While we are unable to establish that this set has \emph{natural} density $0$, we do prove that it is a thin set in a weaker sense.

Our main result is the following.
\begin{thm}\label{thm:b0ConcGen}
Let $f_1,f_2: \mb{N} \ra \mb{C}$ be multiplicative functions and let $h \geq 1$. Assume that
% following conditions:
\begin{align}
\sum_{p: |f_1(p)| \neq 1} \frac{1}{p} = \infty. \label{eq:thinF1}
%&\sum_{p: f(p) = 0} \frac{1}{p} < \infty. \label{eq:thinF0}
\end{align}
Let $a,b \in \mb{N}$, $c,d \in \mb{Z}$ with $ad-bc \neq 0$ and $(a,b) = (c,d) = 1$. Let also $C \in \mb{C} \bk \{0\}$. 
Then the set
$$
\{n \in \mb{N} : f_1(an+b) = Cf_2(cn+d) \neq 0\}
$$
has logarithmic density $0$, i.e., we have
\begin{equation}\label{eq:logDensZeroGen}
\frac{1}{\log x} \sum_{\ss{n \leq x \\ f_1(an+b)f_2(cn+d) \neq 0}} \frac{1_{f_1(an+b) = Cf_2(cn+d)}}{n} = o(1).
\end{equation}
%Then for $u \in \{-1,+1\}$ the sets 
%$$
%\{n \in \mb{N} : f(n) = uf(n+h) \neq 0\}
%$$
%both have logarithmic density 0, i.e., for both $u \in \{-1,+1\}$ we have
%\begin{equation}\label{eq:logDensZero}
%\lim_{x \ra \infty} \frac{1}{\log x} \sum_{\ss{n \leq x  \\ f(n)f(n+h) \neq 0}} \frac{1_{f(n) = uf(n+h)}}{n} = 0.
%\end{equation}
\end{thm}
\begin{rem}
Ideally, we would like to prove the same result with the notion of ``logarithmic density'' replaced by ``natural density''. Unfortunately, we are (unconditionally) limited in this matter by the fact that the linchpin of our proof, Tao's theorem on logarithmically-averaged binary correlations of multiplicative functions (see Theorem \ref{thm1:Tao} below) has not been proven for Ces\`{a}ro-averaged binary correlations. The corresponding result is, however, expected to hold by a conjecture due to Elliott (see \cite{Tao} for a discussion), and assuming Elliott's conjecture our result would immediately be upgraded in this way. 

While a Ces\`{a}ro-averaged result of this kind is known at \emph{almost every scale} $x$ (see \cite[Cor. 1.13]{TT}), this variant would be incompatible with our objectives, since in order to prove Theorem \ref{thm:b0ConcGen} we already need to restrict our attention to an arbitrary infinite subsequence of scales which may not intersect with the ``good'' scales of such a theorem.
\end{rem}
\begin{rem}
The non-vanishing condition $f_1(an+b)f_2(cn+d) \neq 0$ is necessary\footnote{We thank the anonymous referee for pointing out this simple construction.} in this theorem. Indeed, take any complex-valued multiplicative function $f$ and define $\tilde{f} := \mu^2 f$. Then $\tilde{f}$ is supported on squarefree integers. For any fixed, distinct primes $p_1,p_2$ not dividing $ac(ad-bc)$ the set
$$
\{n \in \mb{N} : \tilde{f}(an+b) = \tilde{f}(cn+d)\}
$$
contains the positive density set
$$
\{n \in \mb{N} : p_1^2|(an+b), \quad p_2^2|(cn+d)\}.
$$
For any $n$ in this subset both $an+b$ and $cn+d$ are \emph{not} squarefree, and hence $\tilde{f}(an+b) = 0 = \tilde{f}(cn+d)$. 
\end{rem}
%each element of this latter set yields $f(n) = f(p_1^\nu)f(n/p_1^{\nu})
\begin{rem}
When $f_1,f_2$ are completely multiplicative, the condition $(a,b) = (c,d) = 1$ may be dropped, as we may always reduce to the coprime case (by adjusting the value of $C$). 
%Indeed, if $D_1 = f_1((a,b))$ and $D_2 = f_2((c,d))$ then by setting $a' = a/D_1$, $b' = b/D_1$,
%and $u' = uD_1$, 
%and similarly 
%$c' = c/D_2$, $d' = d/D_2$ and $C' := C D_2/D_1$ we find
%and $v' = vD_2$, we find
%$$
%f_1(an+b) = Cf_2(cn+d) \Leftrightarrow f_1(a'n+b') = C'f_2(c'n+d').
%$$
\end{rem} 
Theorem \ref{thm:b0ConcGen} can be compared with the main result of \cite{EllKish}. There, Elliott and Kish prove that if $f: \mb{N} \ra \mb{C}$ is completely multiplicative and
$$
f(an+b) = Cf(cn+d) \neq 0 \text{ for \emph{all} } n \geq n_0,
$$
where $ad-bc \neq 0$ and $C \in \mb{C} \bk \{0\}$, then there is a Dirichlet character $\chi$ of some modulus $m = m(a,b,c,d)$ such that 
$$
f(p) = \chi(p) \text{ for all } p \nmid m.
$$ 
When $f_1 = f_2=f$ our result addresses the ``$1\%$ world'' version of their problem, showing that if the (logarithmic) proportion of $n$ that satisfy the equation is $\geq \delta$, for $\delta > 0$ a small parameter, then 
\begin{equation} \label{eq:contraF1}
%|f(p)| = 1 \text{ for all but a set of } p. 
\sum_{p : |f(p)| \neq 1} \frac{1}{p} < \infty.
\end{equation}
While this is consistent with the result in \cite{EllKish}, in our setting no further structural information about $f$ can possibly be gleaned, and in particular $f$ need not behave like a Dirichlet character, or even be pretentious (see Section \ref{subsec:pret} for the relevant definition)! Indeed, it is known for instance \cite[Cor. 1.7]{Tao} that the M\"{o}bius function $\mu$ satisfies 
$$
\mu(n) = \mu(n+1) \neq 0 \text{ on a set of logarithmic density } \frac{1}{2}\prod_p (1-2/p^2),
$$
and it is well-known that $\mu$ is non-pretentious. In fact, 
%this conclusion is the best that can be obtained in general, 
%is the best that can be proved. is needed in Theorem \ref{thm:b0ConcGen}, 
we may obtain the following (elementary) converse for real-valued $f$; for simplicity we constrain ourselves to a single case of affine maps $an+b$ and $cn+d$.
\begin{prop} \label{prop:partConv}
Let $f: \mb{N} \ra \mb{R}\bk\{0\}$ be a multiplicative function for which \eqref{eq:contraF1} holds. 
%and suppose in addition that $f(2) = 1$. 
Then there is a $u \in \{-1,+1\}$ 
%and a constant $c_f > 0$ 
such that
$$
\limsup_{x \ra \infty} \frac{1}{\log x}\sum_{\ss{n \leq x}} \frac{1_{f(n) = uf(n+2)}}{n} > 0.
$$
%There are multiplicative functions $f: \mb{N} \ra \mb{C}$ for which \eqref{eq:thinF1} does not hold, and a $u \in \{-1,+1\}$ such that 
%and suppose in addition that $f(2) = 1$. 
%Then there is a $u \in \{-1,+1\}$ 
%and a constant $c_f > 0$ 
%such that
%$$
%\limsup_{x \ra \infty} \frac{1}{\log x}\sum_{\ss{n \leq x \\ f(n)f(n+2) \neq 0}} \frac{1_{f(n) = uf(n+2)}}{n} > 0.
%$$
%(b) There are multiplicative functions $f: \mb{N} \ra \mb{Z}$ for which $\{n \in \mb{N} : f(n) = f(n+2)\}$ does not have logarithmic density $0$.
\end{prop}

Theorem \ref{thm:b0ConcGen} allows us to recover the result
\begin{equation}\label{eq:logDensaPhi}
\lim_{x \ra \infty} \frac{1}{\log x} \sum_{\ss{n \leq x \\ a_{\phi}(n)a_{\phi}(n+h) \neq 0}} \frac{1_{a_{\phi}(n) = a_{\phi}(n+h)}}{n} = 0
\end{equation}
%\eqref{eq:logDensaPhi} 
%$$
for the sequence of Fourier coefficients $(a_{\phi}(n))_n$ of any holomorphic cusp form without CM for $\text{SL}_2(\mb{Z})$.
%as it is well-known that these sequences consist of algebraic integers in some number field. Moreover, 
While this is weaker than what is proved in \cite{heckeKM}, \eqref{eq:logDensaPhi} may be shown \emph{without} appealing to an effective version of the Sato-Tate theorem such as that found in \cite{Tho}. As such, we may avoid having to apply the deep theorems of Newton-Thorne \cite{NeTh} on the functoriality of symmetric power lifts of holomorphic cusp forms. %Moreover, in the case of forms with (rational) integer-valued Fourier coefficients, 
Indeed, our proof (see Section \ref{sec:TauApp}) requires nothing more than:
\begin{itemize}
\item
% an estimate of Serre \cite{Ser} on the sparseness of the set of primes
%$$
%\{p : a_{\phi}(p) = 0\},
%$$
%as well as 
the prime number theorem for Rankin-Selberg $L$-functions, and 
\item Deligne's bound for the prime-indexed Fourier coefficients $|a_{\phi}(p)|$,
\end{itemize}
both of which apply (at least conjecturally and in some cases unconditionally) to the sequence of coefficients of the standard $L$-function of a broader collection of cuspidal automorphic forms.
%As a consequence, our result extends more generally to the coefficients $(a_\pi(m))_m$ of the standard $L$-function of any cuspidal automorphic representation $\pi$ for $GL_m(\mb{A})$ with unitary central character, writing $\mb{A}$ to denote the ad\`{e}les over $\mb{Q}$.
%More precisely, write 
%This allows a treatment, without any appeal to the Sato-Tate theorem, of such functions as the Ramanujan\footnote{Recall that the Ramanujan $\tau$-function is the arithmetic function defined so that the discriminant modular form $\Delta(z)$ satisfies
%$$
%\Delta(z) = q\prod_{n \geq 1} (1-q^n)^{24} = \sum_{n \geq 1} \tau(n)q^n, \quad q = e^{2\pi i z}, \quad \text{Im}(z) > 0.
%$$
%} $\tau$ function, as well as the sequence $(a_E(n))_n$, where 
%$$
%E: y^2 = x^3 + Ax + B
%$$ 
%is an elliptic curve over $\mb{Q}$ and $a_E$ is the multiplicative function defined at primes by
%$$
%a_E(p) := p+1 - |E(\mb{F}_p)|,
%%|\{(x,y) \in (\mb{Z}/p\mb{Z})^2 : y^2 \equiv f(x)^3 \pmod{p}\}|. 
%$$
%and $E(\mb{F}_p)$ is its group of $\mb{F}_p$-points.
%The following result about the Ramanujan $\tau$-function is an easy corollary of Theorem \ref{thm:b0Conc}.
\begin{cor} \label{cor:RamTau}
Let $h\geq 1$ be fixed
%. Let $\pi$ be a fixed cuspidal automorphic representation for $\text{GL}_m(\mb{A})$ with unitary central character, and let $(a_{\pi}(m))_m$ denote the sequence of coefficients of the standard $L$-function of $\pi$.
and let $\phi_1,\phi_2$ be holomorphic cusp forms for the full modular group. 
Then the set
$$
\{n \in \mb{N} : a_{\phi_1}(n) = a_{\phi_2}(n+j) \neq 0 \text{ for some } 1 \leq |j| \leq h\}
$$ 
has logarithmic density 0. 
%\\
%More generally, the same result holds for the sequence of Fourier coefficients $(a_{\phi}(n))_n$ of any non-CM holomorphic cusp form $\phi$ for $SL_2(\mb{Z})$.
\end{cor} 
\begin{rem}
Note that the claim is trivial if, say, $\phi_1$ is assumed to have CM. In this case, it is known \cite{Ser} that
$$
|\{n \leq X : a_{\phi_1}(n) \neq 0\}| \asymp \frac{X}{\sqrt{\log X}},
$$
which necessarily implies that $\{n \in \mb{N} : a_{\phi_1}(n) = a_{\phi_2}(n+h) \neq 0\}$ has logarithmic density $0$. 
\end{rem}
\begin{rem}
A result like Corollary \ref{cor:RamTau} may be proved for the sequences of coefficients $(a_{\pi_j}(n))_n$ of the standard $L$-functions of fixed cuspidal automorphic representations $\pi_j$ for $\text{GL}_{m_j}$, $m_j \geq 1$, with unitary central character ($j = 1,2$), assuming that at least one of $\pi_1,\pi_2$ satisfies the Generalised Ramanujan Conjecture (see e.g. \cite[Sec. 1.3.1]{ManDB} for relevant preliminaries). We leave this extension to the interested reader.
\end{rem}

\subsection{Proof Ideas} \label{sec:proofIdeas}
\subsubsection{Proof strategy towards Proposition \ref{prop:partConv}}
We prove Proposition \ref{prop:partConv} in Section \ref{sec:TrivSieve} by studying the equation $|f(n)| = |f(n+2)|$ for \emph{squarefree} $n$ and $n+2$, noting that since $f$ is real-valued this implies that one of $f(n) = f(n+2)$ or $f(n) = -f(n+2)$ holds. Now, as $f$ is multiplicative and $n$ and $n+2$ are squarefree,
$$
|f(n)| = |f(n+2)| \text{ whenever } |f(p)| = 1 \text{ for all } p|n(n+2).
$$
Writing $\mc{P}_{\neq 1} := \{p : |f(p)| \neq 1\}$, it thus suffices to show that the set\footnote{For a set $S \subset \mb{N}$ we write $(n,S) = 1$ if $(n,m) = 1$ for all $m \in S$.}
$$
\{n \in \mb{N} : (n(n+2),\mc{P}_{\neq 1}) = 1\}
$$
has positive density. This follows straightforwardly from a zero-dimensional sieve, using crucially the fact that $\sum_{p \in \mc{P}_{\neq 1}} p^{-1} < \infty$.
%In the case that we drop the condition $f(n) = f(n+h) \neq 0$ it is easy to construct a multiplicative function 

\subsubsection{Proof strategy towards Theorem \ref{thm:b0ConcGen}} 
%of Theorem \ref{thm:b0Conc} easily in the next subsection, using a zero-dimensional sieve. It is more so the first claim of Theorem \ref{thm:b0Conc} (and its generalisation, Theorem \ref{thm:b0ConcGen}) that is of primary interest here. 
The proof of Theorem \ref{thm:b0ConcGen} is the main goal of the paper. 
%To prove it we use the following method, which we refer to here as the ``projection method''.
We will assume for the sake of contradiction the existence of a (small) parameter $\delta \in (0,1)$ and an infinite increasing sequence $(x_j)_j$ such that
\begin{equation}\label{eq:converseCond}
\frac{1}{\log x_j} \sum_{\ss{n \leq x_j \\ f_1(an+b)f_2(cn+d) \neq 0}} \frac{1_{f_1(an+b) = Cf_2(cn+d)}}{n} \geq \delta, \quad j \geq 1.
\end{equation}
%for all $j \geq 1$.
%We may always decrease $\delta$ without strengthening this assumption, so we may assume it is smaller than any fixed constant. 
We seek a contradiction by obtaining from \eqref{eq:converseCond} precise information about the sequence $(|f(p)|)_p$. 

In order to analyse the condition \eqref{eq:converseCond} we 
%would like to invoke some recently developed tools in the theory of correlations of multiplicative functions. Of particular significance is 
appeal to a deep theorem of Tao \cite{Tao} (see Theorem \ref{thm1:Tao} below) on logarithmically-averaged correlations of multiplicative functions $g_1,g_2$ taking values in $\mb{U}$, the closed unit disc. This result implies structural information about the prime-indexed sequences $(g_j(p))_{p}$, $j = 1,2$, whenever 
%the correlation sums
$$
\lim_{x \ra \infty} \frac{1}{\log x} 
\left|
\sum_{n \leq x} \frac{g_1(an+b)g_2(cn+d)}{n}
\right| \neq 0.
$$
%do not vanish as $x \ra \infty$. 
To apply this we ``project'' our multiplicative functions $f_1,f_2$ onto $\mb{U}$ in a manner that retains their multiplicativity. 
%We observe, crucially, that as $f(n) \in \mc{O}_K$ for all $n$, $N_K(f)(n) := N_K(f(n)) \in \mb{Z}$. Moreover, by taking norms of both sides we have the immediate implication
%$$
%u f(an+b) = vf(cn+d) \Rightarrow U N_K(f)(an+b) = V N_K(f)(cn+d),
%$$
%where $U := N_K(u)$, $V := N_K(v)$, so that $U,V \in \mb{Z} \bk \{0\}$. Thus, upon replacing $f$ by $N_K(f)$, which is still multiplicative, and $u,v$ by $U,V$ respectively, the condition \eqref{eq:converseCond} still holds, and we can reduce immediately to the case that $K = \mb{Q}$ and $f$ takes values in $\mb{Z}$ (hence why we assume \eqref{eq:thinF1} for $|N_K(f(p))|$, rather than $|f(p)|$).

For ease of exposition, 
%for the remainder of this section we will 
suppose that our affine linear forms are $(an+b,cn+d) = (n,n+h)$, $h \geq 1$, that $C = 1$ and that $f_1 = f_2 = f$. 
%Assume also for the moment that $f(n)\neq 0$  for all $n$. 
We observe that for $t \in \mb{R}$ 
%and $\chi$ a Dirichlet character modulo $q \geq 1$ 
the composition
$$
|f|_t(n) := \begin{cases} |f(n)|^{it} &\text{ if $f(n) \neq 0$} \\ 0 &\text{ if $f(n) = 0$} \end{cases}
% \text{ and } f^{\chi}(n) := \chi(f(n))
$$
is a well-defined multiplicative function that takes values in $\mb{U}$. \\
%The collections $\{f_t\}_{t \in \mb{R}}$ 
%and $\{f^{\chi}\}_{\chi \pmod{q}}$ 
%can be used to study the distribution distinct aspects of the value distribution of $(f(n))_n$, as we discuss below. \\ \\
%\textbf{The archimedean argument.}
Let $A \geq 1$ be a parameter to be chosen later. The function
$$
g(n) := 
%\begin{cases} 
A\log |f(n)| 
%&\text{ if $f(n) \neq 0$} \\ 0 &\text{ if $f(n) = 0$} \end{cases} 
$$
is well-defined on the set $\{n \in \mb{N} : f(n) \neq 0\}$, and is additive when restricted to that set. In particular, we have
$$
f(n) = f(n+h) \neq 0 \Rightarrow g(n) = g(n+h),
$$
and also\footnote{Here we use the standard notation $e(y) := e^{2\pi i y}$ for $y \in \mb{R}$.} $|f|_{2\pi At}(n) = e(t g(n))$. Using a Fourier analytic identity, we thus show that if \eqref{eq:converseCond} holds then there is a bounded set $X \subset \mb{R}$ with positive Lebesgue measure such that for each $y \in X$ we have lower bounds 
$$
\frac{1}{\log x_k} \left|\sum_{n \leq x_k} \frac{|f|_{2\pi A y}(n) |f|_{-2\pi A y}(n+h)}{n}\right| \gg_{\delta} 1.
$$
For each of these $y \in X$, Tao's theorem implies the existence of a Dirichlet character $\chi_y$ of conductor depending only on $\delta$, and $t_y \in \mb{R}$ of controlled size such that $f_y$ pretends to be a twisted character $\chi_y(n) n^{it_y}$ (see Section \ref{subsec:pret} for a relevant definition). By passing to a positive measure subset $X'$ of $X$ if needed we may assume that $\chi_y$ is the same for all $y$, and so we focus on studying the dependence of $t_y$ on $y$. 

Using an idea going back to Hal\'{a}sz \cite{Hal} and subsequently refined by Ruzsa \cite{Ruz}, we show that the mapping $y \mapsto t_y$ extends to an entire interval $I$ and is uniformly approximated on $I$ by a linear function $y \mapsto ry$ for some $r \in \mb{R}$ depending at most on $\delta$.  In this way we find
%in light of \eqref{eq:CauchyIdea}, 
that for most primes $p \leq x_k$ and for all $y \in I$, 
$$
|f(p)|^{2\pi i A y} \approx p^{2\pi it_y} \approx p^{2\pi i r y},
$$
and thus that $|f(p)|^A \approx p^r$ for ``typical'' primes $p$. Since $A$ is arbitrarily large, we deduce that $|f(p)| = 1$ outside of a sparse set of primes, which contradicts our assumption \eqref{eq:thinF1}.\\
\begin{rem}
In the recent work \cite{ManInhom}, we considered the variant problem of bounding the size of sets 
%of solutions $n \in \mb{N}$ satisfying
$$
\{n \leq x : f(n+a) = f(n) + b\}, \text{ where } ab \neq 0, \, x \geq 1,
$$
for $f$ an \emph{integer-valued} multiplicative function. We used this together with the work in the present paper to obtain a partial classification of all unbounded multiplicative functions $f: \mb{N} \ra \mb{N}$ for which the set
$$
\{n \in \mb{N} : |f(n+1)-f(n)| \leq C\},
$$
i.e., the set of $n$ for which the gaps between consecutive values of $f(n)$ is bounded by some $C > 0$, has positive upper density.

Unlike in the \emph{homogeneous} (i.e., $a = 0$) problem considered in this paper, examples do exist  where the number of such solutions has positive density (e.g., take $a = b = 1$, then $f(n) = n$ is such an example). 
%We thus seek to classify those functions for which this rigid behaviour occurs. \\
%While some parts of the arguments given here are still applicable in the context of that problem, 
In addition to the ``archimedean'' projections $|f(n)|^{it}$ applied here, we also consider corresponding ``non-archimedean'' projections $\chi(f(n))$, for $\chi$ a Dirichlet character to a suitably chosen modulus $q \geq 1$. The application of these latter projections were developed in connection to the present paper, but thanks to a neat observation by the anonymous referee these turned out to be unnecessary here.
%the non-archimedean part is not directly so, due to the fact that the map
%$$
%n\mapsto \chi(f(n)+a)
%$$
%is \emph{not} multiplicative whenever the conductor of $\chi$ does not divide $a$. In particular, Tao's theorem does not apply, and further work is needed to deal appropriately with the correlations of these functions that, though not multiplicative, are closely related to multiplicative functions.
\end{rem}
%implies that if $g: \mb{N} \ra \mb{C}$ is a multiplicative function with $|g(n)| \leq 1$ for all $n$, then whenever $a,b,c,d \in \mb{Z}$ are fixed integers with $a,c \geq 1$ and $ad \neq bc$, the logarithmically-averaged correlations
% $$
% \frac{1}{\log x} \sum_{n \leq x} \frac{g(an+b)\bar{g}(cn+d)}{n} \ra 0
% $$
% as $x \ra \infty$ \emph{unless} there 

\section{Background Results about Multiplicative Functions}
In this section we collect some background notions and results about multiplicative functions and pretentious number theory to be used in the sequel. 
\subsection{Properties of the pretentious distance}\label{subsec:pret}
A key r\^{o}le is played by the pretentious distances (in the sense of Granville and Soundararajan). Precisely, for arithmetic functions $f,g : \mb{N} \ra \mb{U}$ and $x \geq 2$, define
$$
\mb{D}(f,g;x) := \left(\sum_{p \leq x} \frac{1-\text{Re}(f(p)\bar{g}(p))}{p}\right)^{1/2}, \quad
%$$
%We also write $
\mb{D}(f,g;\infty) := \lim_{x \ra \infty} \mb{D}(f,g;x).
$$
% for $f,g : \mb{N} \ra \mb{U}$ arithmetic functions, and $x \geq 2$. 
Clearly, $0 \leq \mb{D}(f,g;x)^2 \leq 2\log\log x + O(1)$ by Mertens' theorem. When $\mb{D}(f,g;\infty) < \infty$ (so that the upper bound $2\log\log x + O(1)$ is far from the truth) we say that $f$ \emph{pretends to be} $g$ (symmetrically, $g$ \emph{pretends to be} $f$), or that $f$ is \emph{$g$-pretentious}. \\
The pretentious distance functions satisfy the \emph{pretentious triangle inequality}, which we will use in the following two forms \cite[Lem. 3.1]{GSPret}. For functions $f,g,h, f_1,f_2,g_1,g_2: \mb{N} \ra \mb{U} 
$, we have
\begin{align}
&\mb{D}(f,h;x) \leq \mb{D}(f,g;x) + \mb{D}(g,h;x), \label{eq:usualTri}\\
&\mb{D}(f_1f_2,g_1g_2;x) \leq \mb{D}(f_1,g_1;x) + \mb{D}(f_2,g_2;x).\label{eq:multTri}
\end{align}
Applying \eqref{eq:multTri} inductively, we can show that for $f,g : \mb{N} \ra \mb{U}$ and $m \geq 1$,
\begin{equation}\label{eq:pretTrimthPow}
\mb{D}(f^m,g^m;x) \leq m\mb{D}(f,g;x).
\end{equation}
This result implies, for instance, that if $f$ pretends to be $g$ then $f^m$ pretends to be $g^m$. The latter is especially helpful when, say, $g$ takes its non-zero values in roots of unity of some order  (as is the case for instance for Dirichlet characters), in which case $g^m = |g|$, and thus $f^m$ pretends to be a function taking values $0$ or $1$ (and thus is, in a precise sense, ``close'' to being non-negative). \\
Note that when $f(p),g(p) \in S^1$ we may rewrite the numerators of the summands in the definition of $\mathbb{D}(f,g;x)$ using the relation
$$
|f(p)-g(p)|^2 = 2(1-\text{Re}(f(p)\bar{g}(p))).
$$
In particular, it follows that if $f$ is $g$-pretentious and $\eta > 0$ then the set 
$$
\mc{G}_{\eta} := \{p : |f(p)-g(p)| > \eta\}
$$
satisfies
$$
\sum_{\ss{p \leq x \\ p \in \mc{G}_{\eta}}} \frac{1}{p} \leq \eta^{-2} \sum_{p \leq x} \frac{|f(p)-g(p)|^2}{p} = \eta^{-2}\mb{D}(f,g;x)^2 \ll_{\eta} 1.
$$
%We will make conclusions of this kind throughout the paper.
% Note that in contrast to his theorem, our functions are not bounded, and in fact there are no restrictions to the growth of the function at all! In the next two subsections we will introduce devices to allow us to ``project'' the values of $f$ onto the unit disc without losing the multiplicativity of $f$. \\
For convenience, in the sequel we will refer to a subset of primes $S$ as being \emph{sparse} (respectively \emph{$C$-sparse}) if 
$$
\sum_{\ss{p \leq x \\ p \in S}} \frac{1}{p} \ll 1,
$$
(respectively, $\ll_C 1$) as $x \ra \infty$.

The Archimedean characters $n \mapsto n^{it}$, $t \in \mb{R}$, form an important class of examples of multiplicative functions with non-zero mean values and correlations (see Theorem \ref{thm1:Tao} below, for example). In many results in the sequel we will need control of the distance between $n^{it}$ and $1$. This is provided by the following standard result.
% to consider the minimal distances $\inf_{|t| \leq T} \mb{D}(f,n^{it};x)$, which controls the size of the mean values and correlations of a multiplicative function $f$. When $f$ takes non-zero values in bounded order roots of unity, the following result will allow us to assume, effectively, that the minimal distance, if small, arises when $t = 0$.

% %The following lemma will be used to show 
% \begin{lem}\label{lem:removeT}
% Let $1 \leq k,T \leq X$ and let $f: \mb{N} \ra \mu_k \cup \{0\}$, such that $\sum_{p : f(p) = 0} p^{-1} < \infty$. Then
% $$
% \inf_{|t| \leq T} \mb{D}(f,n \mapsto n^{it};X) \geq \frac{1}{2k} \min\{\sqrt{\log_2 x}, \mb{D}(f,1;X)\} - O_k(1).
% $$
% \end{lem}
% \begin{proof}
% This is implicit from the statement and proof of \cite[Lem. 3.1]{KMOrb}.
% \end{proof}

\begin{lem} \label{lem:nitTo1}
Let $\e > 0$ and let $x \geq x_0(\e)$. If $10 \leq |t| \leq x^2$ then
\begin{equation}\label{eq:largeTD}
\mb{D}(n^{it},1;x)^2 \geq (1/3-\e) \log\log x.
\end{equation}
Moreover, if $|t| \leq 10$ then
\begin{equation}\label{eq:smallTD}
\mb{D}(n^{it},1;x)^2 = \log(1+|t|\log x) + O(1).
\end{equation}
\end{lem}
\begin{proof}
Since we could not locate a proof of this precise statement in the literature, we give one here. \\
Set $\sg := 1 + 1/\log x$. Note that for any $t \in \mb{R}$, Mertens' theorem implies that
$$
\mb{D}(1,n^{it};x)^2 = \log\log x - \sum_{p \leq x} \frac{\text{Re}(p^{it})}{p} + O(1) = \log\log x - \log|\zeta(\sg + it)| + O(1),
$$
When $|t| \leq 10$ the Laurent expansion of $\zeta$ near $s = 1$ yields
%the bound \eqref{eq:smallTD} follows from this combined with the standard estimate
$$
\zeta(\sg + it) = \frac{1}{\sg-1 + it} + O(1+|\sg + it-1|) = \frac{\log x}{1 + it\log x} + O(1),
%, \text{ valid for $|t| \leq 10$.}
$$
so that since $(1+|t| \log x)^2 \leq 2(1+t^2\log^2 x) \leq 2(1+|t| \log x)^2$ we find
$$
\log|\zeta(\sg+it)|  = \log\log x - \frac{1}{2}\log(1+t^2\log^2 x) + O(1) = \log\log x - \log(1+|t|\log x) + O(1).
$$
This proves \eqref{eq:smallTD} in the case $|t| \leq 10$. \\
Next, suppose that $10 \leq |t| \leq x^2$. Set $y := V_{x^2} \geq V_t$, where
$$
V_{t} := \exp((\log^{2/3}(2+|t|) \log\log^{1/3}(2+|t|)), \quad t \in \mb{R}.
$$
For $\text{Re}(s) > 1$ define
$$
\zeta_y(s) := \prod_{p > y} \left(1-\frac{1}{p^s}\right)^{-1}.
$$
Combining \cite[Lem. 3.2]{Kou} and \cite[Lem. 4.1]{Kou} (with $\chi$ the trivial character), we find that
$$
\text{Re}\left(\sum_{y < p \leq x} \frac{1}{p^{1+it}}\right) = \log|\zeta_y(\sg + it)|  + O(1) = O(1).
$$
Therefore, we find
$$
\mb{D}(1,n^{it};x)^2 \geq \sum_{y < p \leq x} \frac{1-\text{Re}(p^{-it})}{p} = \log\left(\frac{\log x}{\log y}\right) + O(1) = \frac{1}{3}\left(\log\log x - \log\log\log x\right) + O(1),
$$
and the claim follows.
%whenever $y \geq V_t$. In particular, 
%
%
%
%On the other hand, \eqref{eq:largeTD} follows from the bounds
%$$
%|\zeta(\sg + it)|, |\zeta(\sg + it)|^{-1}  \ll \log^{2/3}(2+|t|) \log\log^{1/3}(2+|t|) \text{ for $|t| \geq 10$,}
%$$
%which are consequences of the Vinogradov-Korobov exponential sum estimates \cite[Cor. 8.28, Thm. 8.29]{IK}.
\end{proof}
%\texttt{Add reference.}

\subsection{Tao's theorem on logarithmically-averaged correlations} 
As described in the previous section, a key tool in our analysis of the distribution of pairs $(f(an+b),f(cn+d))$ is the following theorem of Tao \cite{Tao}, which prescribes necessary conditions for two multiplicative function $g_j: \mb{N} \ra \mb{U}$, $j = 1,2$, to have large, logarithmically-averaged binary correlations.
%of multiplicative functions. The following is a contrapositive form of his result.
\begin{thm}[Tao's theorem on binary correlations] \label{thm1:Tao}
Let $\eta > 0$, and let $a,c \in \mb{N}$, $b,d \in \mb{Z}$ with $ad-bc \neq 0$. Let $g_1,g_2: \mb{N} \ra \mb{U}$ be 1-bounded multiplicative functions with the property that for some $x \geq x_0(\eta)$ we have
$$
\frac{1}{\log x}\left|\sum_{n \leq x} \frac{g_1(an+b)g_2(cn+d)}{n}\right| \geq \eta.
$$
Then for each $j = 1,2$ there is a real number $t_j = t_{j,x}$ with $t_j = O_{\eta}(x)$ and a primitive Dirichlet character $\psi_j = \psi_{j,x}$ with conductor $O_{\eta}(1)$ such that
$$
\sum_{p \leq x} \frac{1-\mathrm{Re}(g_j(p)\bar{\psi}_j(p)p^{-it_j})}{p} = O_{\eta}(1).
% \quad j = 1,2.
$$
\end{thm}

\section{The archimedean projections and the proof of Theorem \ref{thm:b0ConcGen}}
% in the ``real'' aspect}
%using the fact that $\int_{\mb{R}} (\sin(\pi \alpha)/\pi\alpha)^2 d\alpha = 1$.
% If we now assume that the LHS is $\geq \delta$ at a scale $x$ then, choosing $B \geq \delta/2$ we obtain that
% $$
% \max_{|\alpha| \leq B} \left|\frac{1}{\log x} \sum_{n \leq x} \frac{f(n)^{2\pi i \alpha} f(n+1)^{-2\pi i \alpha}}{n}\right| \geq \frac{\delta}{4}.
% $$
% Let $\alpha_0$ maximize the LHS in this last inequality. By Tao's theorem, we find that there is a primitive character $\psi = \psi_{\alpha_0,\delta}$ of conductor $O_{\delta}(1)$ and $t = t_{\alpha_0,\delta} = O_{\delta}(x)$ such that 
% $$
% \sum_{p \leq x} \frac{|f(p)^{2\pi i \alpha_0} - \psi(p)p^{it}|^2}{p} \ll_{\delta} 1.
% $$
% We can argue a bit more carefully and get more information. More precisely, we will prove the following. 
%As discussed in Section \ref{sec:proofIdeas}, 
%While Theorem \ref{thm:b0ConcGen} deals only with multiplicative functions taking algebraic integer values, we will deal in this section with more general complex-valued functions, for which the same arguments can be applied. 

Using Tao's theorem as a principal tool, we will prove the following.
%which applies to all complex-valued multiplicative functions.
\begin{prop}\label{prop:approxCauchyApp}
Let $f_1,f_2: \mb{N} \ra \mb{C}$ be multiplicative functions, and suppose $f_1$ satisfies
\begin{equation}\label{eq:01Both}
\sum_{p : f_1(p) = 0} \frac{1}{p} < \infty, \quad \sum_{p : |f_1(p)| \neq 1} \frac{1}{p} = \infty.
\end{equation}
Let $a,c \in \mb{N}$, $b,d \in \mb{Z}$ and $C \in \mb{C}$, with $C(ad - bc) \neq 0$. Suppose there is a $\delta > 0$ and an infinite increasing sequence $(x_j)_j$ such that for each $j \geq 1$,
\begin{equation}\label{eq:lowerBdAssump}
\frac{1}{\log x_j} \sum_{\ss{n\leq x_j \\ f_1(an+b)f_2(cn+d) \neq 0}} \frac{1_{f_1(an+b) = Cf_2(cn+d)}}{n} \geq \delta.
\end{equation}
Then if $j \geq j_0(\delta)$ and $A \geq 1$ there is a parameter $r_j = r_j(A) \in [-x_j^2,x_j^2]$ such that if $\delta^{-2} \leq B \ll_{\delta} 1$ then
%there is a parameter $r \ll_{\delta} 1$ such that the following holds. Let $\delta^{-2} \leq B \ll_{\delta} 1$
%and define 
%$$
%\mc{S}_{f_1,r,B} := \{p : |f_1(p)|/p^r \notin [e^{-2/B}, e^{2/B}]\}.
%$$ 
%Then the set $S_{f_1,r,B}$ is $O_{\delta}(1)$-sparse, i.e.,
\begin{equation}\label{eq:fullSumwithr}
\sum_{\ss{p \leq x_j  \\ |f_1(p)|^A/p^{r_j} \notin [e^{-2/B}, e^{2/B}]}} \frac{1}{p} \ll_{\delta} 1,
\end{equation}
the bound \eqref{eq:fullSumwithr} being independent of $A$.
\end{prop}
%\begin{rem}
%As the proof shows, while the bounds given depend on $a,b,c,d$ being fixed, they are uniform in $C$.
%are immaterial.
%\end{rem}
\begin{rem} \label{rem:thinF0red}
%The assumption that \eqref{eq:thinF0} holds in Proposition \ref{prop:partConv} is necessary. 
As is suggested by the hypotheses of Proposition \ref{prop:approxCauchyApp}, we will be able to assume in the sequel that
\begin{align}
\sum_{p: f_1(p) = 0} \frac{1}{p} < \infty, \label{eq:thinF0}
\end{align}
as otherwise Theorem \ref{thm:b0ConcGen} is essentially trivial. \\
Indeed, as the function $1_{f_1(n) \neq 0}$ is multiplicative, it follows
%from an elementary theorem due to Hall \cite{Hall} 
that
$$
\sum_{n \leq x} \frac{1_{f_1(n) \neq 0}}{n} \ll \prod_{\ss{p \leq x}} \left(1 + \frac{1_{f_1(p) \neq 0}}{p}\right) \ll (\log x) \exp\left(\sum_{p \leq x} \frac{1_{f_1(p) \neq 0} - 1}{p}\right)
= o(\log x)
$$
whenever $\sum_{p : f_1(p) = 0} p^{-1} = \infty$. This implies that $\{n: f_1(an+b) = Cf_2(cn+d) \neq 0\}$ has logarithmic density $0$ for any fixed $a,b,c,d,C$ as in the theorem.
% in contrast to the conclusion of Proposition \ref{prop:partConv}.
%Delange
%We may recover the same result with the equation $f(n) = f(n+1)$, but in such a case we require the additional necessary condition $f(2) = 1$. Indeed, suppose $f$ is completely multiplicative, satisfying $f(2) = m$ for some $m > 1$, and $f(p)$ is chosen so that $(f(p),m) = 1$ for all $p \geq 3$. Then whenever $n$ is even, $m|f(n)$ but $(m ,f(n+1)) = 1$. Thus, $f(n) = f(n+1)$ has no solutions.
%On the other hand, this condition is not needed in the case that \eqref{eq:thinF1} holds. 
%It may be worthwhile to classify all examples of $f$ for which $f(2) \neq 1$ and yet Proposition \ref{prop:partConv} is still valid. \\
% The above obstruction is due to the fact that the affine linear maps $n$ and $n+1$ cover all residue classes modulo $2$ for all $n$. If, instead, we seek solutions to the equation $f(n) = f(n+h)$ for $2|h$ then this obstruction disappears, and the same arguments can be used \emph{without} requiring any conditions on $f$ besides \eqref{eq:thinF1}. We leave the details to the interested reader.
%The condition \eqref{eq:thinF0} is also needed, for if not then almost all integers $n$ and $n+1$ would have a prime factor $p$ with $f(p) = 0$, and therefore $f(n) = 0$ and $f(n+1)$ for almost all $n$. This would then imply that \eqref{eq:logDensZero} does not hold.
\end{rem}

Theorem \ref{thm:b0ConcGen} follows swiftly from Proposition \ref{prop:approxCauchyApp}.
\begin{proof}[Proof of Theorem \ref{thm:b0ConcGen}]
Let $f_1,f_2: \mb{N} \ra \mb{C}$ be multiplicative functions such that \eqref{eq:thinF1} holds. Assume for the sake of contradiction that
$$
\{n \in \mb{N} : f_1(an+b) = Cf_2(cn+d) \neq 0\}
$$
\emph{does not} have logarithmic density $0$, for some $a,c \in \mb{N}$, $b,d \in \mb{Z}$ with $ad\neq bc$, $(a,b) = (c,d) = 1$, and $C \in \mb{C} \bk \{0\}$. Then there is a $\delta > 0$ and a sequence of scales $(x_j)_j$ such that for each $j \geq 1$,
$$
\frac{1}{\log x_j} \sum_{\ss{n \leq x_j \\ f_1(an+b)f_2(cn+d) \neq 0}} \frac{1_{f_1(an+b) = Cf_2(cn+d)}}{n} \geq \delta,
$$
i.e., \eqref{eq:lowerBdAssump} holds. By Remark \ref{rem:thinF0red}, it suffices to assume also that $f_1$ satisfies \eqref{eq:thinF0}, and thus the conditions in \eqref{eq:01Both} both hold by assumption. \\
Now, let $A \geq 1$. 
%Since $|F_1(p)| \in \{0,1\}$ iff $|f_1(p)| \in \{0,1\}$, \eqref{eq:01Both} also holds with $F_1$ in place of $f_1$, and moreover \eqref{eq:lowerBdAssump} holds with the condition
%$$
%f_1(an+b) = Cf_2(cn+d) \neq 0 \text{ replaced by } F_1(an+b) = C^A F_2(cn+d) \neq 0.
%$$
By Proposition \ref{prop:approxCauchyApp}, for each $j \geq j_0(\delta)$ there exists $r_j  = r_j(A)$ satisfying $|r_j| \leq x_j^2$ such that if we define
$$
I_{A,j} := [e^{-2/(AB)}, e^{2/(AB)} x_j^{x_j^2/A}]
$$
%, $0 \leq r_A \leq R$ and a set $\mc{S}_{A,r,B}$ such that
then it follows that
$$
\Sigma_{A,j} := \sum_{\ss{p \leq x_j \\ |f_1(p)| \notin I_{A,j}}} \frac{1}{p} \leq \sum_{\ss{p \leq x_j \\ |f_1(p)|/p^{r_j/A} \notin [e^{-2/(AB)},e^{2/(AB)}]}} \frac{1}{p} \ll_{\delta} 1,
%\leq \sum_{\ss{p : |F_1(p)|/p^{r_A} \notin [e^{-2/B},e^{2/B}]}} \frac{1}{p} < \infty,
$$
where the bound is independent of $A$. Since $I_{A',j} \subseteq I_{A,j}$ whenever $A' \geq A$, it follows that $A \mapsto \Sigma_{A,j}$ is non-decreasing and uniformly bounded in $A$, for each $j$. Thus, the limit $\Sigma_j := \lim_{A \ra \infty} \Sigma_{A,j}$ exists for each $j \geq j_0(\delta)$, and since moreover
$$
\bigcap_{A \geq 1} I_{A,j} = \{1\}
$$
we deduce that
$$
\Sigma_j = \sum_{\ss{p \leq x_j \\ |f_1(p)| \neq 1}} \frac{1}{p} \ll_{\delta} 1.
$$
Since $(\Sigma_j)_j$ is also bounded and non-decreasing, it converges as $j \ra \infty$, and thus
$$
\sum_{p : |f_1(p)| \neq 1} \frac{1}{p} = \lim_{j \ra \infty} \Sigma_j < \infty,
$$
%the bound only depending on $\delta$. Since $A \mapsto \Sigma_A$ is non-decreasing and $\Sigma_A = O_{\delta}(1)$, we deduce, as $A \ra \infty$, that
%$$
%\sum_{p : |f(p)| \neq 1} \frac{1}{p} = \lim_{A \ra \infty} \Sigma_A < \infty,
%$$
%%
%%It follows that there are $0 < c_1,c_2 = O_{\delta}(1)$ such that
%%$$
%%c_1^{1/A} p^{r_A/A} \leq |f_1(p)| \leq  c_2^{1/A} p^{r_A/A}
%%$$
%%except on a $O_{\delta}(1)$-sparse set. Since $r_A = O_{\delta}(1)$ uniformly in $A$, taking $A \ra \infty$ we find that
%%$$
%%|f_1(p)| = 1 \text{ outside of a $O_{\delta}(1)$-sparse set},
%%$$
which contradicts \eqref{eq:thinF1}. The claim of the theorem thus follows.
\end{proof}
%Now since $r_A = O_{\delta}(1)$ there is a limit point, say $r$, in the set $(r_A)_{A \geq 1}$. If $(A_j)_j$ is a sequence for which $r_A \ra 

\subsection{Reduction to correlations of $1$-bounded functions}
We incorporate the maps $|f|_t := |f|^{it}$ in proving the following, which brings into play lower bounds for correlations of multiplicative functions taking values in $\mb{U}$.

\begin{lem}\label{lem:toBinCorrelft}
Let $A \geq 1$, $\delta \in (0,1)$ and let $B\geq \delta^{-2}$. Let $x$ be a positive real number, sufficiently large with respect to $\delta$, for which
\begin{equation}\label{eq:EqualCond}
\frac{1}{\log x} \sum_{\ss{n \leq x \\ f_1(an+b)f_2(cn+d) \neq 0}} \frac{1_{f_1(an+b) = Cf_2(cn+d)}}{n} \geq \delta.
\end{equation}
Define $F_j := f_j^A$ for $j = 1,2$. Then the set
\begin{equation}\label{eq:largeCorrelft}
X = X_{x,\delta,A,B} := \Bigg\{\alpha \in [-B,B] : \frac{1}{\log x} \left|\sum_{\ss{n \leq x \\ f_1(an+b)f_2(cn+d) \neq 0}} \frac{|F_1|_{\alpha}(an+b)|F_2|_{-\alpha}(cn+d)}{n}\right| \geq \frac{\delta}{16\pi} \Bigg\}
\end{equation}
has Lebesgue measure $\geq \delta/(16\pi)$.
\end{lem}
\begin{proof}
Let $X$ denote the set in \eqref{eq:largeCorrelft}. Clearly, $X$ is symmetric and contains 0. 
The equality $f_1(an+b) = Cf_2(cn+d)$ implies that $|F_1(an+b)| = C'|F_2(cn+d)|$, where $C' := |C|^A$. We next relax this equality to an \emph{inequality}. Borrowing an idea from Davenport and Heilbronn \cite{DavHeil}, we observe that for $t \in \mb{R}$,
\begin{equation}\label{eq:FejerFourierPair}
\max\{1-|t|,0\} = \int_{-\infty}^\infty e(t\alpha) \left(\frac{\sin(\pi \alpha)}{\pi \alpha}\right)^2 d\alpha.
\end{equation}
Crucially, this latter expression is absolutely integrable, with integral  $\leq 1$. Writing $g_j(n) := \log |F_j(n)|$, well-defined whenever $f_j(n) \neq 0$, and $\rho := \log |C'| \in \mb{R}$, we have the trivial implication
\begin{align*}
f_1(an+b) = Cf_2(cn+d) \neq 0 &\Rightarrow |g_1(an+b)-g_2(cn+d)-\rho| \leq 1/2,
\end{align*}
from which it follows immediately that
\begin{equation} \label{eq:indicatorUppBd}
1_{f_1(an+b) = Cf_2(cn+d) \neq 0} \leq 2\max\{1-|g_1(an+b)-g_2(cn+d) - \rho|, 0\}.
\end{equation}
Combining \eqref{eq:indicatorUppBd} with \eqref{eq:FejerFourierPair} and inserting the result  into \eqref{eq:EqualCond}, we obtain
\begin{align*}
\delta &\leq \frac{1}{\log x} \sum_{\ss{n \leq x \\ f_1(an+b)f_2(cn+d) \neq 0}} \frac{1_{f_1(an+b) = Cf_2(cn+d)}}{n} \\
&\leq 2 \int_{-\infty}^\infty \left(\frac{1}{\log x} \sum_{\ss{n \leq x \\ f_1(an+b)f_2(cn+d) \neq 0}} \frac{e(g_1(an+b)\alpha) e(-g_2(cn+d)\alpha)}{n}\right) e(-\rho\alpha) \left(\frac{\sin(\pi \alpha)}{\pi \alpha}\right)^2 d\alpha \\
&= 2 \int_{-\infty}^\infty \left(\frac{1}{\log x} \sum_{\ss{n \leq x \\ f_1(an+b)f_2(cn+d) \neq 0}} \frac{|F_1|_{2 \pi \alpha}(an+b)|F_2|_{-2 \pi \alpha}(cn+d)}{n}\right) e(-\rho\alpha)\left(\frac{\sin(\pi \alpha)}{\pi \alpha}\right)^2 d\alpha.
\end{align*}
%We have thus obtained an upper bound for the sum we intend to estimate, involving binary correlations of 1-bounded multiplicative functions. 
Bounding the range $|\alpha| > B$ trivially, we obtain after a change of variables $\alpha \mapsto 2\pi \alpha$ that
\begin{align}
%&\frac{1}{\log x} \sum_{\ss{n \leq x \\ f(an+b)f(cn+d) \neq 0}} \frac{1_{rf(an+b) = sf(cn+d)}}{n}\nonumber \\
\delta &\leq 4\pi \int_{-B}^B \left(\frac{1}{\log x} \sum_{\ss{n \leq x \\ f_1(an+b)f_2(cn+d) \neq 0}} \frac{|F_1|_{\alpha}(an+b) |F_2|_{-\alpha}(cn+d)}{n}\right) e^{-i\rho\alpha} \left(\frac{\sin(\alpha/2)}{ \alpha/2}\right)^2 d\alpha + O\left(\frac{1}{B}\right). \label{eq:truncDH}
%&\leq 2 \max_{|\alpha| \leq B} \left|\frac{1}{\log x} \sum_{n \leq x} \frac{f(n)^{2\pi i \alpha} f(n+1)^{-2\pi i \alpha}}{n}\right| + O(1/B),
\end{align}
Recall that $B \geq \delta^{-2}$ and note that the integrand above is bounded by $1$, 
From the definition \eqref{eq:largeCorrelft} of $X$ and \eqref{eq:truncDH} we derive when $\delta$ is small enough that
$$
\frac{\delta}{2} \leq 4\pi \int_X \left|\frac{1}{\log x} \sum_{\ss{n \leq x \\ f_1(an+b)f_2(cn+d) \neq 0}} \frac{|F_1|_{\alpha}(an+b) |F_2|_{-\alpha}(cn+d)|}{n}\right| \left(\frac{\sin(\alpha/2)}{\alpha/2}\right)^2 d\alpha + \frac{\delta}{4} \leq 4\pi\lambda(X) + \frac{\delta}{4},
$$
where $\lambda(X)$ denotes the Lebesgue measure of $X$. Rearranging, we obtain $\lambda(X) \geq \delta/(16\pi)$, as claimed. 
\end{proof}
As a corollary of this and Tao's theorem, we deduce the following. 
%For convenience, in the sequel we will employ the convention $0^{it} := 0$ for all $t \in \mb{R}$.
\begin{cor} \label{cor:structuredX}
Assume the notation of Lemma \ref{lem:toBinCorrelft}, and let $f: \mb{N} \ra \mb{C}$ be the multiplicative function given at prime powers by $f(p^k) = 1$ if $f_1(p^k) = 0$, otherwise $f(p^k) = F_1(p^k)$. Then there are positive integers $m, L = O_{\delta}(1)$, independent of $A$,
% and a subset $X' \subseteq X$ with $\lambda(X') \gg_{\delta} 1$ 
such that for each
%\footnote{For a set $A \subseteq \mb{R}$ and $t \in \mb{R}$ we use the notation  $t \cdot A := \{t a : a \in A\}$.} 
$\beta \in X$ there is a real number $t_{\beta} \in [-Lx,Lx]$ such that
 $$
 \mb{D}(|f^m|_{\beta}, n^{im t_{\beta}};x) = O_{\delta}(1),
 $$
and the upper bound is independent of $A$.
\end{cor}

\begin{proof}
%Assume $j$ is chosen sufficiently large as a function of $\delta$, which is chosen as small as required (without loss of generality we can always do this). We write $x = x_j$. \\
Write $\mc{P}_{f_1} := \{p^k : f_1(p^k) = 0\}$, so that $f(p^k) = F_1(p^k)$ for all $p^k \notin \mc{P}_{f_1}$. Let $h(n)$ be the multiplicative function given at prime powers by $h(p^k) := 1-1_{\mc{P}_{f_1}}(p^k)$. 
%that denote by $(m,\mc{P}_f) = 1$ the condition that if $p^k||m$ then $p^k \notin \mc{P}_f$. 
Combining Lemma \ref{lem:toBinCorrelft} with Theorem \ref{thm1:Tao}, applied to the multiplicative function
$g = |f|_{\alpha}h = |F_1|_{\alpha}h,$
we deduce that there is $L = O_{\delta}(1)$ such that for each $\alpha \in X$ there is a primitive character $\psi_\alpha$ of conductor $O_\delta(1)$ and a $t_\alpha \in [-Lx,Lx]$ such that
\begin{align*}
\mb{D}(|f|_{\alpha}, \psi_{\alpha}(n)n^{it_{\alpha}})^2 &= \sum_{p \leq x} \frac{1-\text{Re}(|f(p)|^{i\alpha} \bar{\psi}_{\alpha}(p)p^{-it_\alpha})}{p} \\
&= \sum_{p \leq x} \frac{1-\text{Re}(|F_1(p)|^{i\alpha} h(p) \bar{\psi}_{\alpha}(p)p^{-it_\alpha})}{p} + O(1) \ll_{\delta} 1.
\end{align*}
Now let $\nu = O_{\delta}(1)$ be the maximum of the conductors of all the $\psi_{\alpha}$, and set $m := \nu!$.
% As there are only $O_{\delta}(1)$ primitive characters of conductor $\ll_{\delta} 1$, by the pigeonhole principle we may pick a fixed character $\psi$ of conductor $O_{\delta}(1)$ and order $m = O_{\delta}(1)$, and choose $X' \subseteq X$ with $\lambda(X') \gg_{\delta} \lambda(X) \gg_{\delta} 1$ such that $\psi_\alpha = \psi$ for all $\alpha \in X'$. Let $m$ be the order of $\psi$, 
Then $\psi_{\alpha}^m(p) = 1$ for all $p\nmid m$. By the pretentious triangle inequality \eqref{eq:pretTrimthPow}, uniformly over $\alpha \in X$ we have
$$
\mb{D}(|f^m|^{i\alpha}, n^{imt_\alpha};x)^2 
%= \sum_{p \leq x} \frac{1-\text{Re}(|f(p)|^{im\alpha}p^{-imt_\alpha})}{p} 
\leq m^2 \mb{D}(|f|_{\alpha}, \psi_{\alpha}(n)n^{it_{\alpha}};x)^2
%\sum_{p \leq x} \frac{1-\text{Re}(|f(p)|^{i\alpha} \bar{\psi}(p)p^{-it_\alpha})}{p} 
+ O_{\delta}(1) \ll_{\delta} 1.
$$
%uniformly over $\alpha \in X'$. 
Note that the above bound depended only on the validity of \eqref{eq:EqualCond} and \eqref{eq:01Both}, both of which are independent of $A$, as required.
% of the corollary follows. 
%from this after making the change of variables $\beta = m\alpha \in m \cdot X'$.
\end{proof}

\subsection{An application of the approximate Cauchy functional equation} \label{sec:CauchyFE}
We seek to understand the structure of the map $\beta \mapsto t_{\beta}$ arising in Corollary \ref{cor:structuredX}. In this direction we will use an idea of Hal\'{a}sz \cite{Hal}, improved upon in a paper by Ruzsa \cite{Hal}, which will allow us, roughly speaking, to approximate this map by a linear map $\beta \mapsto c\beta$, $c \in \mb{R}$, on some appropriate \emph{sumset} of $\tilde{X} := m \cdot X'$. 
%In certain ways this is the continuous analogue of the treatment of large order characters in my paper. 
The proof is based on a few lemmas. Here and in the sequel, given subsets $T_1,T_2$ of an abelian group $(G,+)$, we define the sumset
$$
T_1+T_2 := \{t_1+t_2 : t_1 \in T_1, t_2 \in T_2\} \\
$$ and for $m \geq 2$ the iterated sumset
$$
mT := \{t_1 + \cdots + t_m : t_i \in A  \text{ for all } 1 \leq i \leq m\}.
$$
In the sequel, for $S \subseteq \mb{R}$ and $\lambda > 0$ we will also denote by $\lambda \cdot S$ the dilation
$$
\lambda \cdot S := \{\lambda t : t \in S\}.
$$
% For $A \subseteq \mb{R}$ and $t \in \mb{R}$ we also denote
% $$
% t\cdot A := \{ta : a \in A\}.
% $$
%The following two lemmas from analysis will allow us to 
% , essentially due to Hal\'{a}sz, shows that the sumsets of certain  subsets $X \subseteq \mb{R}$ of positive measure contain full intervals (generalising a well-known observation due to Steinhaus).  
\begin{lem}\label{lem:dilateRuzsa}
Let $D > 0$ and let $Y \subseteq [-D,D]$ be a symmetric\footnote{We say that a set $T \subseteq \mb{R}$ is \emph{symmetric} if $t \in T$ if and only if $-t \in T$.}  set containing 0, and denote its Lebesgue measure by $\lambda(Y)$. Then if $\ell \geq \llf 12D/\lambda(Y)\rrf$ the sumset $\ell Y$ contains $[-D,D]$. 
\end{lem}
\begin{proof} 
Setting $Y' := \tfrac{1}{D}\cdot Y \subseteq [-1,1]$, we have $\lambda(Y') = \lambda(Y)/D$. Applying \cite[Lemma (5.3)]{Ruz} to $Y'$, we get that $\ell Y'$ contains $[-1,1]$ as soon as $\ell \geq \llf 12/\lambda(Y') \rrf = \llf 12 D/\lambda(Y)\rrf$. The claim then follows, as $y' = y'_1 + \cdots + y'_\ell$ holds with $y'_j \in Y'$ if and only if $Dy' = y_1 + \cdots + y_\ell$ with $y_j \in Y = D \cdot Y'$. 
\end{proof}
\begin{lem}\label{lem:CauchyRuz}
Let $K \geq 0$. Let $\phi: \mb{R} \ra \mb{R}$ be a bounded function, and suppose that whenever $\alpha_1,\alpha_2,\alpha_1+\alpha_2 \in [-B,B]$ we have
$$
|\phi(\alpha_1+\alpha_2) - \phi(\alpha_1) - \phi(\alpha_2)| \leq K.
$$
Then there is a $c \in \mb{R}$, specifically $c = \phi(B)/B$, with
$$
\sup_{\alpha \in [-B,B]} |\phi(\alpha)-c\alpha| \leq 3K.
$$
\end{lem}
\begin{proof} 
Defining $\psi(\beta) := \phi(B\beta)/B$ for $\beta \in [-1,1]$, this follows directly from applying \cite[Lem. 5.11]{Ruz} to $\psi$ and rescaling by $B$. %and using homogeneity of the expressions given here.
\end{proof}

%In the sequel, we will say that two real numbers $y_1,y_2$ are $(\delta,x)$-close if $|y_1-y_2|\leq \delta/\log x$. 
\begin{proof}[Proof of Proposition \ref{prop:approxCauchyApp}]
Let $\delta \in (0,1)$ and suppose $f: \mb{N} \ra \mb{C}$ satisfies \eqref{eq:lowerBdAssump} for $a,c \in \mb{N}$, $b,d \in \mb{Z}$ and $C \in \mb{C}$ and $C(ad- bc) \neq 0$. Thus, there is a sequence $(x_j)_j$ such that for $j \geq 1$ and $x = x_j$, \eqref{eq:EqualCond} holds. Fix $j$ large for the moment. By Corollary \ref{cor:structuredX}, there is a set $X \subseteq [-B,B]$ and an integer $m = O_{\delta}(1)$
% and a map $\phi_0 : X' \ra \mb{R}$ (given by $\phi_0(\alpha) := t_{\alpha}$ in the notation of Corollary \ref{cor:structuredX}) satisfying $\sup_{\alpha \in X'} |\phi_0(\alpha)| \leq Ax_j$ for some $A = O_{\delta}(1)$, 
such that for each $\beta \in X$ there is $t_{\beta} = t_{\beta}(x_j)$ with
\begin{equation}\label{eq:defForTildeX}
\mb{D}(|f^m|_{\beta}, n^{imt_\beta};x_j) = O_{\delta}(1),
\end{equation}
where we recall that $f(p^k) = F_1(p^k) = f_1(p^k)^A$ whenever $f_1(p^k) \neq 0$. The bound here is independent of $A$. \\
%and
%(alternatively, we could replace $X'$ by its difference set $X'-X'$), 
%Note that $\phi_0(0) = O_{\delta}(1/\log x_j)$ by Lemma \ref{lem:nitTo1}. 
%
%Set $D := mB$ and $\tilde{X} = m \cdot X \subseteq [-D,D]$, and define $\phi_j(\beta) := m\phi_0(\beta/m)$ on $\tilde{X}$. We extend $\phi_j$ to $[-D,D]$ so that
%\begin{equation}\label{eq:phiMinimiser}
%\text{if } \beta \in [-D,D] \bk \tilde{X} \text{ then } \mb{D}(|f|_{\beta}, n^{i\phi_j(\beta)};x_j) = \min_{|t| \leq Ax_j} \mb{D}(|f|_{\beta}, n^{it};x_j)
%\end{equation}
%(if, for any such $\beta$ more than one minimiser exists, we assign $\phi_j(\beta)$ to be any of them).
% $ we define $\phi(\beta) := mt_{\beta/m}$. Otherwise, for any $\beta \in \mb{R} \bk \tilde{X}$ we let $\phi(\beta)$ be a minimiser on $[-Ax,Ax]$ for the map $t \mapsto \mb{D}(|f|^{i\beta},n^{it};x)$, where $A \asymp_{\delta}1$ is chosen large enough so that $|t_{\beta}| \leq Ax$ for all $\beta \in \tilde{X}$ (this can be done since the upper bound on the size of $t_{\beta}$ is determined uniformly as a result of Tao's theorem). Clearly, then, $\phi$ is a bounded function. \\
Lemma \ref{lem:dilateRuzsa} shows that when $k = \lceil 100 B/\lambda(X)\rceil \asymp_{\delta} 1$ the sumset $k X$ contains $[-B,B]$. In particular, every $\beta\in [-B,B]$ may be written as
\begin{equation}\label{eq:sumsetCond}
\beta= \alpha_1 + \cdots + \alpha_k, \quad \alpha_\ell \in X.
\end{equation}
Define $t_{\beta} := t_{\alpha_1}+\cdots + t_{\alpha_k}$ in this case (if more than one such decomposition exists, pick any of them). This defines a mapping $\beta \mapsto t_{\beta}$ on all of $[-B,B]$, such that $|t_{\beta}| \leq k A L x_j < x_j^{3/2}$ uniformly, provided $x_j$ is sufficiently large with respect to $\delta$. By the pretentious triangle inequality,
$$
\mb{D}(|f^m|_{\beta}, n^{imt_{\beta}};x_j) = \mb{D}(|f^m|^{i(\alpha_1 + \cdots + \alpha_k)}, n^{im(t_{\alpha_1} + \cdots + t_{\alpha_k})};x_j) \leq  \sum_{\ell=1}^k \mb{D}(|f^m|_{\alpha_\ell}, n^{imt_{\alpha_\ell}};x_j) \ll_{\delta} 1.
$$
Moreover, if $\beta_1,\beta_2,\beta_1+\beta_2 \in [-B,B]$ then
\begin{align*}
\mb{D}(n^{imt_{\beta_1+\beta_2}}, n^{im (t_{\beta_1} + t_{\beta_2})};x_j) &\leq \mb{D}(|f^m|_{\beta_1+\beta_2}, n^{imt_{\beta_1 + \beta_2}};x_j) + \mb{D}(|f^m|_{\beta_1} n^{imt_{\beta_1}};x_j) + \mb{D}(|f^m|_{\beta_2},n^{imt_{\beta_2}};x_j) \\
&\ll_{\delta} 1.
\end{align*}
%Provided $x_j$ is sufficiently large in terms of $\delta$, 
Since $m|t_{\beta_1+\beta_2} - t_{\beta_1} - t_{\beta_2}| = O_{\delta}(x_j^{3/2}) < x_j^2$ for $j$ sufficiently large, Lemma \ref{lem:nitTo1} yields
$$
\log(1+m|t_{\beta_1+\beta_2}-t_{\beta_1} - t_{\beta_2}| \log x_j) \ll_{\delta} 1.
$$
It follows that there is $K = O_{\delta}(1)$ such that for any $\beta_1,\beta_2 \in [-B,B]$ with $\beta_1+\beta_2 \in [-B,B]$,
$$
|t_{\beta_1+\beta_2} - t_{\beta_1} - t_{\beta_2}| \leq \frac{K}{\log x_j}. 
$$
%for some $K = O_{\delta}(1)$. 
By Lemma \ref{lem:CauchyRuz} there is $r_j = r_j(A) \in \mb{R}$ such that uniformly over $\beta \in [-B,B]$,
$$
t_{\beta} = r_j \beta + O_{\delta}\left(\frac{1}{\log x_j}\right);
$$
since $|t_{\beta}| < x_j^{3/2}$, say, for $j$ large enough, we see that $|r_j| < x_j^{2}$. It follows that
$$
p^{im(t_{\beta}-r_j\beta)} = 1 + O_{\delta}\left(\frac{mK \log p}{\log x_j}\right) 
%= 1+O_{\delta}\left(\frac{\log p}{\log x_j}\right) 
\text{ for all } p \leq x_j^{1/(3mK)}.
$$ 
Applying \eqref{eq:pretTrimthPow}, we find that, uniformly over $\beta \in [-B,B]$, 
%given a representation \eqref{eq:sumsetCond} we have
\begin{align*}
%\sum_{p \leq x} \frac{1-\text{Re}(|f(p)|^{i\beta}p^{-ib\beta})}{p} 
\mb{D}(|f^m|_{\beta}, n^{imr_j\beta};x_j)^2 &= \sum_{p \leq x_j^{1/(3mK)}} \frac{1-\text{Re}(|f(p)|^{im\beta}p^{-im t_{\beta}})}{p} + O\left(\sum_{x_j^{1/(3mK)}<p \leq x_j} \frac{1}{p} + \frac{mK}{\log x_j} \sum_{p \leq x_j^{1/(3mK)}} \frac{\log p}{p}\right) 
%&\leq k^2\max_{1 \leq \ell \leq k} \mb{D}(|f^m|_{\beta_\ell},p^{i\phi_j(\beta_\ell)};x_j)^2 + O_{\delta}(1) 
\ll_\delta 1.
\end{align*}
Equivalently, we may write this as
\begin{equation}\label{eq:checkIndepr}
\max_{\beta \in [-B,B]} \sum_{p \leq x_j} \frac{1-\text{Re}((|f(p)|/p^{r_j})^{im\beta})}{p} \ll_{\delta} 1,
\end{equation}
emphasising once again that this bound is independent of $A$.  
%since $j_0 = j_0(\delta)$, note also that
%there is a $J = J(\delta)$ such that if $j \geq J$ then $|r - r_j| \leq 1$. Since 
%$$
%|r_{j_0}| \asymp_{\delta} |r_{j_0}| D \leq |t_D(x_{j_0})| + |t_D(x_{j_0})-r_{j_0}D| \ll_{\delta} 1.
%$$ 
%Thus, in light of \eqref{eq:rjGap} we have $r = O_{\delta}(1)$ as well. \\
%Finally, note that if $p_0$ is the least prime $p$ for which $|f(p)|/p^r \in [e^{-1/B},e^{1/B}]$ then $p_0 = O_{\delta}(1)$, and thus for small enough $\delta$, $2^r \leq p_0^r \leq 3|f(p_0)|/2 \ll_{\delta} 1$. Hence, $r = O_{\delta}(1)$, and is determined by $f$ at primes of size $p = O_{\delta}(1)$. 
% Taking $j \ra \infty$, we obtain
% $$
% \max_{\beta \in [-B,B]} \mb{D}(|f|_{\beta}, n^{ir\beta}; \infty) \leq C(\delta).
% $$
%which proves the first claim of the proposition (aside from the assertion $r = O_{\delta}(1)$, which will be proven momentarily). \\
%For the second claim, 
We may now complete the proof of the proposition. Given \eqref{eq:checkIndepr}, after the change of variables $\beta \mapsto m\beta$ we observe that
$$
\frac{1}{2mB} \int_{-mB}^{mB} \sum_{p \leq x_j} \frac{1-\text{Re}((|f(p)|/p^{r_j})^{i\beta})}{p} d\beta
%= \frac{1}{2B} \int_{-B}^B \mb{D}(|f|_{\beta},n^{ir\beta};X)^2 d\beta 
= O_{\delta}(1).
$$
Since $\frac{1}{2R} \int_{-R}^R y^{it} dt = \frac{\sin(R \log y)}{R\log y} =: \text{sinc}(R\log y)$, we deduce that
$$
\sum_{p \leq x_j} \frac{1-\text{sinc}(mB\log(|f(p)|/p^{r_j}))}{p} = O_{\delta}(1).
$$
As $1-\text{sinc}(t) \gg 1$ whenever $|t| \geq 2$, recalling the notation of the proposition and that $f(p) \neq f_1(p)$ on a sparse set, we deduce using $m \geq 1$ that
$$
\sum_{\ss{p \leq x_j \\ |F_1(p)|/p^{r_j} \notin [e^{-2/B}, e^{2/B}]}} \frac{1}{p} \leq \sum_{\ss{p \leq X \\ |\log(|f(p)|/p^{r_j})| \geq 2/(mB)}} \frac{1}{p} + O(1) = O_{\delta}(1).
$$
%For parameters $0 \leq u \leq 1 < v$ let us define the prime sums
%$$
%L_v(X) := \sum_{\ss{p \leq X \\ |\log(|f(p)|/p^r)| \geq v/B}} \frac{1}{p}, \quad M_{u,v}(X) := \sum_{\ss{p \leq X \\ u/B \leq |\log(|f(p)|/p^r)| < v/B}} \frac{1}{p}.
%$$
%Note that if $|t| \geq 3/2$, say, then $|\text{sinc}(t)| \leq 1/|t|$, while if $|t| < 3/2$ then $\text{sinc}(t) \leq 1-7t^2/16$. Taking $v = 3/2$ and $u = 1$, say, we thus find that
%$$
%\frac{1}{3}L_v(X) \leq \sum_{\ss{p \leq X \\ |\log(|f(p)|/p^r)| \geq 3/(2B)}} \frac{1-\text{sinc}(B\log(|f(p)|/p^r))}{p} \leq \sum_{p \leq X} \frac{1-\text{sinc}(B\log(|f(p)|/p^r))}{p} \ll_{\delta} 1,
%$$
%and similarly
%$$
%\frac{7}{16} M_{1,3/2}(X) \leq \sum_{\ss{p \leq X\\ 1/B \leq |\log(|f(p)|/p^r)| < 3/(2B)}} \frac{1-\text{sinc}(B\log(|f(p)|/p^r))}{p} \ll_{\delta} 1.
%$$
%Recalling the notation of the proposition, 
%%these two statements together imply that
%$$
%\sum_{\ss{p \leq X \\ p \in \mc{S}_{f,r,B}}} \frac{1}{p} \leq M_{1,3/2}(X) + L_{3/2}(X) \ll_{\delta} 1.
%$$
%Taking $X \ra \infty$, we establish \eqref{eq:fullSumwithr}, as claimed.
%It remains to show that $r = O_{\delta}(1)$. To see this, 
Recalling that $F_1 = f_1^A$, the claim follows.
\end{proof}

\section{Proof of the Converse Result} \label{sec:TrivSieve}
%\subsection{The case when $f(p) \neq 1$ occurs on a thin set}
%\begin{lem}\label{lem:thinfn}
\begin{proof}[Proof of Proposition \ref{prop:partConv}]
Let $f: \mb{N} \ra \mb{R} \bk \{0\}$ be multiplicative, satisfying the condition
\begin{equation} \label{eq:primCondsAgain}
%\sum_{p : f(p) = 0} \frac{1}{p} < \infty, \quad \quad 
\sum_{p : |f(p)| \neq 1} \frac{1}{p} < \infty.
\end{equation}
% Then there is a constant 
% $$
% c_f := \prod_{f(p) \neq 1} \left(1-\frac{2}{p}\right) > 0
% $$ 
% such that $|\{n \leq x : f(n) = f(n+1)\}| \geq (1-o_f(1)) c_f x$.
% %\end{lem}
% \begin{proof}
%This is a simple exercise in the inclusion-exclusion principle. 
Observe that if there is $c_f' > 0$ such that
\begin{equation}\label{eq:redGoalConv}
\frac{1}{\log x} \sum_{n \leq x} \frac{1_{|f(n)| = |f(n+2)|}}{n} 
%|\{n \leq x : |f(n)| = |f(n+2)| \neq 0\}| 
\geq c_f' + o(1),
\end{equation}
then by the pigeonhole principle there is a $u_x \in \{-1,+1\}$ such that if $c_f := c_f'/2$ then
$$
\frac{1}{\log x} \sum_{n \leq x} \frac{1_{f(n) = u_xf(n+2)}}{n} 
%|\{n \leq x : |f(n)| = |f(n+2)| \neq 0\}| 
\geq c_f + o(1).
$$
%|\{n \leq x : f(n) = uf(n+2) \neq 0\}| \geq c_f x.
Pigeonholing once again, we can choose an infinite sequence $(x_j)_j$ on which $u = u_{x_j} \in \{-1,+1\}$ is constant, in which case we obtain
$$
\limsup_{x \ra \infty} \frac{1}{\log x} \sum_{n \leq x} \frac{1_{f(n) = uf(n+2)}}{n}  > 0,
$$
as claimed. \\
Moreover, the condition \eqref{eq:primCondsAgain} is unchanged when passing from $f$ to $|f|$. Thus, in the sequel, we replace $f$ by $|f|$ and assume that $f \geq 0$. We will show that \eqref{eq:redGoalConv} holds whenever $x \geq x_0(f)$. \\
By partial summation, it will suffice to show that
$$
\sum_{n \leq x} 1_{f(n) = f(n+2)} \geq (c_f' + o(1)) x
$$
for all $x \geq x_0(f)$. To this end, we restrict to $n \equiv 1 \pmod{2}$, so that $n$ and $n+2$ are coprime, and also such that $n,n+2$ are both squarefree. Observe that for such $n$, if $f(p) = 1$ for all $p||n$ and all $p||n+2$ then $f(n) = 1 = f(n+2)$. Thus, we obtain
\begin{equation}\label{eq:restrictToSfreeetc}
\sum_{n \leq x} 1_{f(n) = f(n+2)} \geq \sum_{\ss{n \leq x  \\ p|n(n+2) \Rightarrow f(p) = 1}} \mu^2(n)\mu^2(n+2)1_{2 \nmid n}.
\end{equation}
Take $z := (\log x)^{0.99}$. As $\sum_{p : f(p) \neq 1} p^{-1} < \infty$, we find
$$
\sum_{\ss{n \leq x \\ \exists p > z: \\
p|n(n+2), \, f(p) \neq 1}} 1  \ll \sum_{\ss{z < p \leq x \\ f(p) \neq 1}} \left(\frac{2x}{p} + 1\right) \ll x\left(\sum_{\ss{p > z \\ f(p) \neq 1}} \frac{1}{p} + \frac{1}{\log x}\right) = o(x).
$$
It follows that
$$
\sum_{\ss{n \leq x \\ p|n(n+2) \Rightarrow f(p) = 1}} \mu^2(n)\mu^2(n+2)1_{2 \nmid n} = \sum_{\ss{n \leq x \\ p|n(n+2), \, p \leq z \Rightarrow f(p) = 1}} \mu^2(n)\mu^2(n+2)1_{2 \nmid n} + o(x).
$$
% Next, set $K := \llf 10 \log\log x \rrf+1$. Observe that the set of $n \leq x$ for which $\nu_p(n(n+1)) > K$ for some prime $p$ is
% $$
% \leq \sum_{k > 10 \log\log x} \sum_{p^k \leq x} \sum_{\ss{n \leq x \\ p^k || n}} 1 \ll  \sum_{2 \leq p \leq x^{1/K}} \sum_{k > K} \left(\frac{x}{p^k}+1\right) \ll x\sum_{2 \leq p \leq x^{1/K}} \frac{1}{p^K} \ll \frac{x}{2^{K-1}} \log\log x \ll \frac{x}{\log x}
% $$
% Thus, to summarise, we have
% $$
% \sum_{n \leq x} 1_{f(n) = f(n+1)} \geq \sum_{\ss{n \leq x \\ p^k || n(n+1), p \leq z \Rightarrow k \leq 10 \log\log x, \\ f(p^k) = 1}} 1.
% $$
If we set
$$
P(z) := \prod_{\ss{3 \leq p \leq z \\ %\exists \, k\geq 1: 
f(p) \neq 1}} p,
$$
then the prime number theorem trivially gives $P(z) \leq x^{0.01}$ for large enough $x$. Thus, for any divisors $d_1,d_2|P(z)$ we have $d_1d_2 \leq x^{0.02}$. Since $2 \nmid P(z)$, by M\"{o}bius inversion we obtain
\begin{align*}
\sum_{\ss{n \leq x \\ p|n(n+2), \, p \leq z \Rightarrow f(p) = 1}} \mu^2(n)\mu^2(n+2)1_{2 \nmid n} &= \sum_{\ss{d_1,d_2 | P(z) \\ (d_1,d_2) = 1}}\mu(d_1)\mu(d_2) \sum_{\ss{n \leq x  \\ n \equiv 0 \pmod{d_1} \\ n + 2 \equiv 0 \pmod{d_2}}} \mu^2(n)\mu^2(n+2)1_{2 \nmid n}.
\end{align*}
By e.g. \cite[Lem. 3.3]{BivEK} (taking $a = 2$) the inner sum in this last expression is, for each pair of coprime divisors $d_1,d_2$ of $P(z)$,
$$
x\frac{\phi(d_1)\phi(d_2)}{2(d_1d_2)^2}\prod_{p \nmid d_1d_2}(1-\frac{2}{p^{2}}) + O_{\e}(x^{2/3+\e}).  
$$
Summing over all such $d_1,d_2$, we thus obtain
$$
Cx\sum_{\ss{d_1,d_2 | P(z) \\ (d_1,d_2) = 1}} \frac{\mu(d_1)\mu(d_2) \phi(d_1)\phi(d_2)}{(d_1d_2)^2} \prod_{p|d_1d_2}\left(1-\frac{2}{p^2}\right)^{-1} + O(x^{0.7}),
$$
where $C := \frac{1}{2} \prod_p (1-2p^{-2})$. As
$$
\sum_{\ss{d \geq 1 \\ p|d \Rightarrow f(p) \neq 1}} \frac{1}{d} \ll \exp\left(\sum_{p : f(p) \neq 1} \frac{1}{p}\right) \ll 1,
$$
due to the trivial bound $\phi(d_1)\leq d_1$ the double series in $d_1,d_2|P(z)$ converges absolutely as $z \ra \infty$, and therefore
\begin{align*}
&C\sum_{\ss{d_1,d_2 | P(z) \\ (d_1,d_2) = 1}} \frac{\mu(d_1)\mu(d_2) \phi(d_1)\phi(d_2)}{(d_1d_2)^2} \prod_{p|d_1d_2}\left(1-\frac{2}{p^2}\right)^{-1} \\
&= C \sum_{\ss{d_1,d_2 \geq 1 \\ p|d_1d_2 \Rightarrow f(p) \neq 1 \\ (d_1,d_2) = 1 \\ 2 \nmid d_1d_2}} \frac{\mu(d_1)\mu(d_2) \phi(d_1)\phi(d_2)}{(d_1d_2)^2} \prod_{p|d_1d_2}\left(1-\frac{2}{p^2}\right)^{-1} + o(1).
\end{align*}
Upon computing Euler products we find the asymptotic estimate
$$
\sum_{\ss{n \leq x \\ p|n(n+2) \Rightarrow f(p) = 1}} \mu^2(n)\mu^2(n+2)1_{2 \nmid n} = (c_f'+o(1))x,
$$
where $c_f'$ is defined by the (convergent) product
$$
c_f' := \prod_{\ss{p \geq 3 \\ f(p) \neq 1}} \left(1-\frac{2(p-1)}{p^2-2}\right) > 0.
$$
Combined with \eqref{eq:restrictToSfreeetc}, this leads to the lower bound
$$
\sum_{n \leq x} 1_{f(n) = f(n+2) \neq 0} \geq (c_f' + o(1)) x,
$$
as claimed.

% = \sum_{\ss{d_1,d_2 | P(z) \\ (d_1,d_2) = 1}}\mu(d_1)\mu(d_2) \left(\frac{x}{d_1d_2} + O(1)\right)\\
% &= x\sum_{\ss{d_1,d_2 |P(z)\\ (d_1,d_2) = 1}} \frac{\mu(d_1)\mu(d_2)}{d_1d_2} + O(x^{1/50}) = x\prod_{\ss{3 \leq p \leq z \\ f(p) \neq 1}}\left(1-\frac{2}{p}\right) + O(x^{1/50}).
% \end{align*}
% Since, for any $T \geq z$ we have 
% $$
% 0 \leq 1-\prod_{\ss{z \leq p \leq T \\ f(p) \neq 1}}\left(1-\frac{2}{p}\right) = 1-\exp\left(\sum_{\ss{z \leq p \leq T \\ f(p) \neq 1}} \log(1-2/p)\right) \ll \sum_{\ss{z \leq p \leq T \\ f(p) \neq 1}}\log(1/(1-2/p)) = o_f(1),
% $$
% and using $f(2) = 1$, we deduce that
% $$
% \sum_{n \leq x} 1_{f(n) = f(n+2)} \geq (1-o_f(1))x\prod_{\ss{p > 2 \\ f(p) \neq 1}} \left(1-\frac{2}{p}\right),
% $$
% and \eqref{eq:redGoalConv} now follows with
% $$
% c_f' := \prod_{\ss{p > 2 \\ f(p) \neq 1}} \left(1-\frac{2}{p}\right) > 0.
% $$ 
\end{proof}

\section{Proof of Corollary \ref{cor:RamTau}} \label{sec:TauApp}
Let $\phi$ be a holomorphic cuspidal eigenform without CM for $\text{SL}_2(\mb{Z})$, of weight $k \geq 2$, normalised so its sequence of Fourier coefficients $(a_{\phi}(n))_n$ satisfies $a_{\phi}(1) = 1$. 
It is well-known that 
%there is a number field $K = K_{\phi}$ such that each $a_{\phi}(n) \in \mc{O}_K$, and 
the map $n \mapsto a_{\phi}(n)$ is multiplicative, so that $f = a_{\phi}$ fits the setup of Theorem \ref{thm:b0ConcGen}.

To prove Corollary \ref{cor:RamTau} it suffices to check that \eqref{eq:thinF1} holds
%\footnote{Note that this may be accomplished without appealing to the Sato-Tate theorem in the case of $\tau$.} 
for $a_{\phi}(n)$, after which the claim then follows from Theorem \ref{thm:b0ConcGen} (applied to each of the equations $a_{\phi}(n) = a_{\phi}(n+j)$, $1 \leq |j| \leq h$).

%We can prove \eqref{eq:thinF0} in a unified way for any $\phi$ using the following consequence of a well-known theorem of Serre \cite[Thm. 15]{Ser}: as $X \ra \infty$,
%$$
%\pi_{\phi}(X) := |\{p \leq X : a_{\phi}(p) = 0\}| \ll \frac{X}{(\log X)^{6/5}}.
%$$
%Thus, by partial summation we find
%$$
%\sum_{\ss{p \leq x \\ a_{\phi}(p) = 0}} \frac{1}{p} \ll \frac{\pi_\phi(x)}{x} + \int_2^x \frac{\pi_{\phi}(y)}{y^2} dy  \ll 1 + \int_2^X \frac{dy}{y(\log y)^{6/5}} \ll 1,
%$$
%verifying \eqref{eq:thinF0}. 
%We will check \eqref{eq:thinF1} separately for $\tau$ first (the same argument works whenever $\phi$ has integer-valued Fourier coefficients). 
Set $\lambda_\phi(n) := a_{\phi}(n)/n^{(k-1)/2}$. By the prime number theorem for Rankin-Selberg $L$-functions (e.g. combining \cite[Cor. 1.2 and Lem. 5.1]{LWY} with \cite[Exer. 6]{IK}),
\begin{equation} \label{eq:RSPNT}
\sum_{p \leq X} |\lambda_{\phi}(p)|^2 \log p = X + O(Xe^{-c\sqrt{\log X}}),
\end{equation}
for some $c = c_{\phi} > 0$. We claim that if $y_0$ is sufficiently large and $y \geq y_0$ then
\begin{equation}\label{eq:lowbdnon1}
\sum_{\ss{y \leq p \leq 2y \\ |a_{\phi}(p)| \neq 1}} \log p \geq y/5.
\end{equation}
Indeed, suppose \eqref{eq:lowbdnon1} fails for some $y \geq y_0$, with $y_0$ large enough. From \eqref{eq:RSPNT} and Deligne's bound $|a_{\phi}(p)| \leq 2p^{(k-1)/2}$ we obtain
$$
\sum_{\ss{y \leq p \leq 2y \\ |a_{\phi}(p)| = 1}} |a_{\phi}(p)p^{-(k-1)/2}|^2 \log p \geq 9y/10 - 4\sum_{\ss{y \leq p \leq 2y \\ |a_{\phi}(p)| \neq 1}} \log p > y/10.
$$
But by the prime number theorem, the left-hand side of this last expression is
$$
\leq y^{1-k} \sum_{\ss{y \leq p \leq 2y}} \log p \ll y^{2-k},
$$
which is impossible for $y_0$ large enough and $k \geq 2$. 

By a dyadic decomposition and \eqref{eq:lowbdnon1}, we find
$$
\sum_{\ss{p \leq x \\ |a_{\phi}(p)| \neq 1}} \frac{1}{p} \geq \sum_{y_0 \leq 2^j \leq x/2} \frac{1}{2^{j+1}\log(2^{j+1})} \sum_{\ss{2^j \leq p \leq 2^{j+1} \\ |a_{\phi}(p)| \neq 1}} \log p \geq \frac{1}{10 \log 2} \sum_{y_0 \leq 2^j \leq x/2} \frac{1}{j+1} = c \log\log x + O(1),
$$
with $c = 1/(10\log 2)$, so \eqref{eq:thinF1} clearly holds as well. 
\section*{Acknowledgments}
The author would like to thank C\'{e}dric Pilatte warmly for his valuable suggestions on a previous version of this paper, which led to a shorter proof of Theorem \ref{thm:b0ConcGen} with wider applicability.

%%% AUTHOR:
%%% Bibliography goes here. Note that the arXiv cannot process bibtex
%%% or biber bibliographies.  Example of acceptable bibliograpy format:
\bibliographystyle{amsplain}

\begin{thebibliography}{99}

\bibitem{DavHeil}
H.~Davenport and H.~Heilbronn.
\newblock On indefinite quadratic forms in five variables.
\newblock {\em J. London Math. Soc.}, 21:185--193, 1946.

\bibitem{EllKish}
P.D.T.A. Elliott and J.~Kish.
\newblock Harmonic analysis on the positive rationals. {D}etermination of the
  group generated by the ratios $(an+b)/(an+b)$.
\newblock {\em Mathematika}, 63(3):919--943, 2017.

\bibitem{EPS2}
P.~Erd\H{o}s, C.~Pomerance, and S\'{a}rk\H{o}zy.
\newblock On locally repeated values of certain arithmetic functions {II}.
\newblock {\em Acta Math. Hung.}, 49:251--259, 1987.

\bibitem{EPS}
P.~Erd\H{o}s, C.~Pomerance, and S\'{a}rk\H{o}zy.
\newblock On locally repeated values of certain arithmetic functions {III}.
\newblock {\em Proc. Amer. Math. Soc.}, 101(1):1--7, 1987.

\bibitem{GSPret}
A.~Granville and K.~Soundararajan.
\newblock Large character sums: pretentious characters and the
  {P}\'{o}lya-{V}inogradov theorem.
\newblock {\em J. Amer. Math. Soc}, 20(2):357--384, 2007.

\bibitem{Hal}
G.~Hal\'{a}sz.
\newblock On the distribution of additive arithmetical functions.
\newblock {\em Acta Arith.}, 27:143--152, 1975.

\bibitem{IK}
H.~Iwaniec and E.~Kowalski.
\newblock {\em Analytic number theory}, volume~53 of {\em American Mathematical
  Society Colloquium Publications}.
\newblock American Mathematical Society, Providence, RI, 2004.

\bibitem{heckeKM}
O.~Klurman and A.P. Mangerel.
\newblock Monotone chains of {F}ourier coefficients of {H}ecke cusp forms.
\newblock arXiv:2009.03225 [math.NT].

\bibitem{Kou}
D.~Koukoulopoulos.
\newblock Pretentious multiplicative functions and the prime number theorem for
  arithmetic progressions.
\newblock {\em Compos. Math.}, 149(7):1129--1149, 2013.

\bibitem{LWY}
J.~Liu, Y.~Wang, and Y.~Ye.
\newblock A proof of {S}elberg's orthogonality for automorphic {$L$}-functions.
\newblock {\em Manuscripta Math.}, 118(2):135--149, 2005.

\bibitem{ManInhom}
A.P. Mangerel.
\newblock Gap problems for integer-valued multiplicative functions.
\newblock arXiv:2311.11636 [math.NT].

\bibitem{BivEK}
A.P. Mangerel.
\newblock On the bivariate {E}rd\"{o}s-{K}ac theorem and correlations of the
  {M}\"{o}bius function.
\newblock {\em Math. Proc. Camb. Phil. Soc.}, 169(3):1--59, 2019.

\bibitem{ManDB}
A.P. Mangerel.
\newblock Divisor-bounded multiplicative functions in short intervals.
\newblock {\em Res. Math. Sci.}, 10:47 pp., 2023.

\bibitem{NeTh}
J.~Newton and J.A. Thorne.
\newblock Symmetric power functoriality for holomorphic modular forms.
\newblock {\em Publ. Math. IHES}, 134:1--116, 2021.

\bibitem{Ruz}
I.~Ruzsa.
\newblock On the concentration of additive functions.
\newblock {\em Acta Math. Acad. Sci. Hung.}, 36:215--232, 1980.

\bibitem{Ser}
J.-P. Serre.
\newblock Quelques applications due th\'{e}or\`{e}me de densit\'{e} de
  chebotarev.
\newblock {\em Pub. Math. de l'IH\'{E}S}, 54:123--201, 1981.

\bibitem{Tao}
T.~Tao.
\newblock The logarithmically-averaged {C}howla and {E}lliott conjectures for
  two-point correlations.
\newblock {\em Forum Math. Pi}, 4 e8, 2016.

\bibitem{TT}
T.~Tao and J.~Ter\"{a}v\"{a}inen.
\newblock The structure of correlations of multiplicative functions at almost
  all scales, with applications to the {C}howla and {E}lliott conjectures.
\newblock {\em Algebra Number Theory}, 13(9):2103--2150, 2019.

\bibitem{Tho}
J.~Thorner.
\newblock Effective forms of the {S}ato-{T}ate conjecture.
\newblock {\em Res. Math. Sci.}, 8(1):1--21, 2020.

\end{thebibliography}

%% AUTHOR: You can generate such a bibliography from a .bib file by 
%% running pdflatex/bibtex/pdflatex/pdflatex and then pasting the .bbl file
%% between \begin{thebibliography} and \end{bibliography}

%%% AUTHOR: Include a short description of each author following the
%%% structure below. Use the same short tags used previously.  
%%% Use \imageat{} and \imagedot{} instead of "@" and "." in
%%% email addresses-this replaces the symbols with graphics to avoid 
%%% e-mail address harvesting from the .pdf file
\begin{dajauthors}
\begin{authorinfo}[APM]
  Alexander P. Mangerel\\
  Department of Mathematical Sciences, Durham University\\ Stockton Road, Durham, DH1 3LE, UK\\
  smangerel\imageat{}gmail\imagedot{}com \\
  \url{https://sites.google.com/site/amangerel/home}
\end{authorinfo}
\end{dajauthors}

\end{document}